\documentclass{amsart}
\usepackage{amsmath}
  \usepackage{paralist}
  \usepackage{graphics} 
  \usepackage{epsfig} 
\usepackage{graphicx}  \usepackage{epstopdf}
 \usepackage[colorlinks=true]{hyperref}

\newtheorem{theorem}{Theorem}[section]
\newtheorem{corollary}{Corollary}

\newtheorem{lemma}[theorem]{Lemma}
\newtheorem{proposition}{Proposition}

\theoremstyle{definition}
\newtheorem{definition}[theorem]{Definition}
\newtheorem{remark}{Remark}

\newtheorem{mtheorem}{Theorem}
\newtheorem{mproposition}[mtheorem]{Proposition}

\def\Inte{\mathop{\mathrm{Int}}}
\def\Cl{\mathop{\mathrm{Cl}}}

\title[Existence of minimal flows on nonorientable surfaces] 
      {Existence of minimal flows on nonorientable surfaces}

\author[J.G. Esp\'in Buend\'ia,  D. Peralta-Salas and G. Soler L\'opez]{}

\subjclass{37C70; 37C10; 37E0.}
 \keywords{Flow, Surface, Minimality, Transitivity, Interval and circle exchange transformations, Suspension}

\email{josegines.espin@um.es}
\email{dperalta@icmat.es}
\email{gabriel.soler@upct.es}
\thanks{$^*$ Corresponding author: josegines.espin@um.es}

\begin{document}
\maketitle

\centerline{\scshape Jos\'e Gin\'es Esp\'in Buend\'ia$^*$}
\medskip
{\footnotesize
 \centerline{Departamento de Matem\'aticas}
\centerline{Universidad de Murcia (Campus de Espinardo)}
\centerline{30100 Espinardo-Murcia (Spain)}
} 

\medskip

\centerline{\scshape Daniel Peralta-Salas}
\medskip
{\footnotesize
 \centerline{ Instituto de Ciencias Matem\'aticas}
 \centerline{Consejo Superior de Investigaciones Cient\'ificas}
  \centerline{28049 Madrid (Spain)}
}

\medskip

\centerline{\scshape Gabriel Soler L\'opez}
\medskip
{\footnotesize
 \centerline{ Departamento de Matem\'atica Aplicada y Estad\'istica}
 \centerline{Universidad Polit\'ecnica de Cartagena }
   \centerline{Paseo Alfonso XIII, 52}
     \centerline{30203 Cartagena (Spain)}
}

\bigskip

 \centerline{(Communicated by the associate editor name)}

\begin{abstract} Surfaces admitting flows all whose orbits are dense are called minimal. Minimal orientable surfaces were characterized by J.C. Beni\`{e}re in 1998, leaving open the nonorientable case. This paper fills this gap providing a characterization of minimal 
nonorientable surfaces of finite genus. We also construct an example 
of a minimal nonorientable surface with infinite genus and conjecture that any nonorientable surface without combinatorial boundary is minimal.
\end{abstract}

\section{Introduction}

The problem of finding transitive flows on manifolds has a long tradition
(see for instance the bibliography in \cite{thomas}). Two works 
contributed to solve the problem for surfaces. Firstly, in $1998$, 
J. C. Beni\`{e}re proved in his Ph.D. Thesis \cite{beniere} that all noncompact orientable 
surfaces which are not embeddable in the real euclidean plane posses a
minimal flow. Independently, in $1999$, the third 
author in his Master Thesis~\cite{tesinasoler} and in~\cite{victorgabi}, with V. Jim\'{e}nez,
characterized all transitive surfaces of finite genus. However, up to our knowledge, 
the minimality of nonorientable surfaces has not been 
characterized so far. The present paper fills this gap for the case of
nonorientable surfaces of finite genus and makes progress in the study of the infinite genus ones.

\def\T{\mathbb{T}}
\def\B{\mathbb{B}}

There is nothing to say about the study of compact minimal surfaces. According to 
the Poincar\'e-Hopf Index Theorem, if a compact surface $S$ admits a smooth minimal flow,
the Euler characteristic of $S$ must be zero and either $S$ is the torus $\T_2$
or the Klein bottle $\B_2$. The latter case can be discarded because $\B_2$ does not
admit nontrivial recurrent orbits. On the other hand, $\T_2$ admits in fact 
an analytic minimal flow because so is the irrational flow.

Thus, it suffices to focus on the study of noncompact surfaces.
No subsurfaces of the sphere, the real projective plane or the Klein bottle can posses minimal flows since 
these surfaces do not admit flows with nontrivial recurrent orbits (see e.g.~\cite[Section~2.2]{rusos}). 
The nonorientable compact surface of genus $3$ (the torus with a cross-cap) and any of its subsurfaces of genus $3$ cannot 
possess minimal flows either. This result is stated in~\cite[p. 14]{rusos} without a proof; for the sake of completeness, we shall provide a proof in Appendix~\ref{ApenA}. 

In the literature, it is possible to find some isolated examples of noncompact nonorientable surfaces of finite genus with
a minimal flow. For instance, in $1978$ C. Gutierrez constructed a minimal flow on the compact nonorientable surface of genus 4 minus
two points (\cite{gutierrez3}). Similarly, it can be checked that, for any positive integer $2+n$, the surface obtained
after removing $n$ points from a compact nonorientable surface of genus $2+n$ admits a minimal flow (\cite{capitulosoler}). 
 
Instead of talking about minimal flows we could indistinguishably use the notion of minimal vector fields. Given a surface $S$ we may say that a smooth vector field $X$ on $S$ is minimal if its associated maximal flow is minimal (i.e. if any maximal integral curve associated to $X$ is dense in $S$). The aim of this work is to complete the characterization of noncompact 
surfaces of finite genus which admit a minimal smooth vector field and to give an example of 
a nonorientable surface of infinite genus with the same property. 

The following is the first main result of the paper. In its statement, we say that $X$ is \textit{area preserving} if there exists a non-degenerate and complete $2$-form $\theta$ (an area form) such that the Lie derivative of $\theta$ with respect to $X$ is equal to zero (i.e. $\mathcal{L}_X(\theta)=0$).

\begin{mtheorem}\label{T:maingenerofinito}
Let $S$ be an orientable noncompact surface of finite genus $g\geq 1$  or a nonorientable noncompact surface of finite
 genus $g\geq 4$. Then there exists a real analytic complete vector field $X$ on $S$ which is minimal and area preserving.
\end{mtheorem}


As already mentioned, the case of orientable surfaces in Theorem~\ref{T:maingenerofinito} was proved by Beni\`{e}re~\cite{beniere}. Nevertheless, in order to make our exposition as self-contained as possible, we also include a proof of that case. 

Beni\`ere's approach for proving Theorem~\ref{T:maingenerofinito} in the orientable case relies on a geometrical method for gluing together different foliated elementary models. Once the pieces are glued, one gets a compact surface endowed with a foliation with only one singularity and whose regular leaves are all dense. Such a foliation has the additional property of admitting a closed transverse curve whose associated Poincar\'e map is an oriented interval exchange transformation. Our proof follows a kind of opposite path. For proving Theorem~\ref{T:maingenerofinito} (for both orientable and nonorientable surfaces), our approach consists in building surfaces and vector fields by suspending certain kind of interval exchange transformations. This idea is usually employed in the literature to get flows on surfaces with different properties. As far as this procedure is concerned, the reason why the case of nonorientable surfaces has remained unsolved up today has to do with 
the enormous difficulty involved in the task of constructing nonorientable minimal interval exchange transformations. Accordingly, the keystone of the proof of Theorem~\ref{T:maingenerofinito} is a recent work by A. Linero and the third author which fully characterizes nonorientable interval exchange transformations all whose orbits are dense (see Theorem~\ref{minlinsol} stated in Section~\ref{Giets}). Suspending an appropriate nonorientable exchange transformation for every $g \geq 4$, we get a minimal flow on the noncompact surface obtained from the nonorientable compact surface of genus $g$ after removing one point. Nonetheless, in order to achieve minimal flows on any nonorientable noncompact surface of finite genus $g$, additional nontrivial work is still needed: one has to remove a Cantor set of points from the compact surface in such a way that the restricted flow is still minimal. This task follows similar ideas to those presented in Beni\`ere's work. 

Generalizing his geometrical method, Beni\`{e}re also established the minimality for orientable surfaces of infinite genus. When dealing with nonorientable surfaces of infinite genus, we can prove the following result. 

\begin{mtheorem}\label{T:maingeneroinfinito}
 There exist nonorientable surfaces of infinite genus which posses minimal area preserving complete real analytic vector fields.
\end{mtheorem}

We emphasize that the proof of Theorem~\ref{T:maingeneroinfinito} is independent of the aforementioned Theorem~\ref{minlinsol}: the construction of such a minimal vector field  is obtained by applying the suspension method to a minimal interval exchange transformation with infinitely many discontinuities. In particular, we prove that:

\begin{mproposition}\label{P:mainiet}
 There exists a minimal interval exchange transformation with flips and with infinitely many points of discontinuity. 
\end{mproposition}
It is worth pointing out that, as far as we know, such an example is new in the literature. We conjecture that a future development in the study of interval exchange transformations will allow us to prove that any nonorientable surface of infinite genus is minimal. 

The content of the paper is organized as follows. In Section~\ref{basic} we recall some preliminary results on surfaces and Cantor sets and present the classification of noncompact surfaces due to I. Richards. In Section~\ref{Giets} we introduce the notion of generalized interval exchange transformation and summarize without proofs some results that will be used later. The proof of Proposition~\ref{P:mainiet} is presented in Section~\ref{infiniteiet}. In Section~\ref{flowgiet} we explain how to construct minimal vector fields by the suspension of interval exchange transformations. Finally, Theorems~\ref{T:maingenerofinito} and~\ref{T:maingeneroinfinito} are proved in Section~\ref{proofmainth}.

\section{Preliminary results: open surfaces and Cantor sets}\label{basic}


In what follows, for every non-negative (respectively positive) integer $g$, $M_g$ (respectively $N_g$) denotes the only, 
up to homeomorphisms, orientable (respectively nonorientable) compact surface
of genus $g$. Recall that the Euler characteristic, $\chi(\cdot)$, can be computed as  $\chi(M_g)=2-2g$ and $\chi(N_g)=2-g$.

Up to $C^r$-diffeomorphisms (with $r= \infty$ or $r=\omega$), any surface (compact or not, with or without boundary) has a unique $C^r$ structure and given any two surfaces they 
are homeomorphic if and only if they are $C^r$-diffeomorphic. In what follows, all the surfaces will be supposed to be equipped with
a compatible analytic structure. 

\def\S{\mathbb{S}}

In~\cite{richards}, I. Richards provided a complete classification of both compact and not compact surfaces. This characterization is set in the so-called \textit{Ker\'ekj\'art\'o Theorem} (see~\cite[Theorem~1, p. 262]{richards}). In the same paper, the following representation theorem for surfaces is stated (see~\cite[Theorem~3, p. 268]{richards}).

\begin{theorem}[Richards]\label{richardsth} Every surface is homeomorphic to a surface formed from the euclidean sphere $\S^2$ by first removing a compact totally disconnected set $K \subset \S^2$, then removing the interior of a finite or infinite sequence $(D_i)_i$ of disjoint closed disks in $\S^2 \setminus K$, and finally suitably identifying the boundaries of these disks in pairs (it may be necessary to identify the boundary of one disk with itself to produce a cross-cap). Moreover, when the sequence $(D_i)_i$ is infinite, for any open subset $U \subset \S^2$ containing $K$, all but a finite number of the $D_i$ are contained in $U$. \end{theorem}

For the case of finite genus surfaces, Theorem~\ref{richardsth} and Ker\'ekj\'art\'o Theorem imply the following:
\begin{corollary}\label{finitegenus}
Let $S$ be an orientable (respectively nonorientable) surface of finite genus $g$. Then, for any compact orientable (respectively nonorientable) surface of genus $g$, $M$, there exists a totally disconnected subset $K \subset M$ such that $M \setminus K$ is homeomorphic to $S$. Even more, if $L$ is a totally disconnected subset of $M$ which is homeomorphic to $K$, then $M \setminus L$ is also homeomorphic to $S$.
\end{corollary}

\begin{remark}\label{exfinite}This corollary shows that every noncompact surface of finite genus posses a compactification which is itself a surface. In general, given a noncompact surface $S$, an embedding $h:S \to M$ of $S$ into a topological space $M$ is said to be a \textit{compactification of $S$} if $M$ is compact and $h(S)$ is an open and dense subset of $M$. Given a noncompact orientable (respectively nonorientable) surface $S$ of finite genus $g$, Corollary~\ref{finitegenus} says that there exists an embedding $h: S \to M$ with $M=M_g$ (respectively $M=N_g$) such that $K=M \setminus h(S)$ is totally disconnected. So, in particular, $h$ is a compactification of $S$. The set $M$ is also locally connected and Hausdorff (it is a surface) and $K$ is nonseparating on $M$ (i. e. for any open connected subset $U \subset M$, the set $U \setminus K$ is also connected). This additional property makes the compactification $h$ be unique in the following sense. If $M'$ is any other compac
 t Hausdorff and locally connected space for which there exists an embedding $h':S \to M'$ making $K'=M' \setminus h'(S)$ being totally disconnected and nonseparating on $M'$, it is standard that there exists a homeomorphism $F:M \to M'$ such that $(h')^{-1} \circ F \circ h$ is the identity map on $S$.
\end{remark}
Given a particular surface $S$, deciding its orientability can be a difficult task. The orientability of $S$ can be proved finding an explicit atlas compatible with the analytic structure of $S$ and for which all the transition maps between non-disjoint coordinate charts have positive Jacobian. A criterion for $S$ being nonorientable is that it contains nonorientable circles. We recall that any surface with boundary homeomorphic to $\mathbb{S} \times [-1,1]$ (respectively 
to $N_1 \setminus U$ where $U$ is the interior of a closed disk $D \subset N_1$) is said to be a \textit{closed annulus} (respectively a  \textit{M\"{o}bius band}). As is well-known, any circle 
in any surface has a neighbourhood whose closure is either a closed annulus (an orientable circle) or a closed M\"{o}bius band (a nonorientable circle).

\def\R{\mathbb{R}}
\def\N{\mathbb{N}}

To conclude this section, we present some properties about Cantor sets which we shall need in the proof of Theorem~\ref{T:maingenerofinito}. A non-empty topological space $K$ is said to be a \textit{Cantor set} if is metrizable, compact, totally disconnected and perfect (i. e. without isolated points). All the Cantor sets are homeomorphic; in particular every Cantor set is homeomorphic to the ternary Cantor set in $\R$ . Even more:
\begin{theorem}\label{T:Cantor} Let $K$ be a Cantor set and fix a point $p_0 \in K$. Then, for any compact, metric, totally disconnected space $L$ there exists an embedding $h: L \to K$ with $p_0 \in h(L)$. 
\end{theorem}
\begin{proof}
This is an elementary consequence of a well-known topological result stating that any compact metric totally disconnected space has a homeomorphic copy inside any Cantor set (see~\cite[p. 285]{ku1}).

If $L$ is itself a Cantor set there is nothing to say: two any Cantor sets are homeomorphic. In the contrary case, $L$ has isolated points; let us fix one of these isolated points $q_0 \in L$. Let $h: L \to K$ be an embedding and suppose that $p_0 \notin h(L)$. It is enough then to consider the map $\tilde{h}: L \to K$ given by $\tilde{h}(q)=h(q)$ if $q \neq q_0$ and $\tilde{h}(q_0)=p_0$. This new map is also injective (because so is $h$), continuous (because $q_0$ is isolated) and closed (all the continuous maps from compact spaces to Hausdorff spaces are closed).\end{proof}

\def\Q{\mathbb{Q}}

We next present a final technical lemma regarding Cantor sets cited in~\cite[pp. 14--15]{beniere}. A detailed proof is given in Appendix~\ref{ApenB}.
A subset $K \subset \R$ is called \textit{rationally independent} if for any non-empty finite subset $F=\{x_1, \ldots, x_k\} \subset K$ there are no integers $n_1, \ldots, n_k$ such that $ n_1 x_1 + \cdots + n_k x_k \in \Q$ and $(n_1)^2 + \cdots  + (n_k)^2 \neq 0$ (i. e. such that they do not vanish simultaneously); $K$ is said \textit{rationally dependent} otherwise.

\begin{lemma}\label{L:cantor} Let $I\subset \R$ be an open interval. Any rationally independent finite set $F \subset I$ is contained in a 
rationally independent Cantor set $K \subset I$.\end{lemma}

\section{Generalized interval exchange transformations}\label{Giets}

In what follows, $\N$ will denote the set of positive integer numbers. Given any $n \in \N$ (respectively $n=\infty$) we shall write $\N_n$ to denote the set $\{1,2,\ldots,n\}$ (respectively $\N$); sometimes, in 
order to unify the  notation, we will also put $\infty=\infty+1=\infty -1$ so in particular $\N_{\infty}=\N_{\infty + 1}=\N$.

\def\cets{c.e.t.'s}
\def\cet{c.e.t.}
\def\iets{i.e.t.'s}
\def\iet{i.e.t.}
\def\Cets{C.e.t.'s}
\def\Cet{C.e.t.}
\def\Iets{I.e.t.'s}
\def\Iet{I.e.t.}
\def\giet{g.i.e.t.}
\def\giets{g.i.e.t.'s}
\def\gcet{g.c.e.t.}
\def\gcets{g.c.e.t.'s}

Let $a<b$ be two real numbers and $D$ be an open subset of $(a,b)$. The connected components of $D$ set a countable family of open intervals of $\R$; that is, there are some $n \in \N \cup \{\infty\}$ and a family of pairwise disjoint open intervals $\{I_i=(a_i,a_{i+1})\}_{i \in \N_n}$ with $D=\cup_{i \in \N_n} {I_i}$. Following \cite{gutierrez4} we say that an injective map $T:D \to [a,b]$ is a \emph{generalized interval exchange transformation in $(a,b)$}, abbreviated as  \giet, if both $D$ and its image $T(D)$ are open and dense subsets of $(a,b)$ and $T$ homeomorphically takes each connected component of $D$ onto a connected component of $T(D)$. 

In this work we will focus only on the family of \giets\ $T : D \to [a,b]$ with the extra property that the restriction of $T$ to any of the components of $D$ is an affine map of constant slope equal to $1$ or $-1$; given a \giet\ $T$ of such a family we will say that it is an \emph{interval exchange transformation (of $n$-intervals)}, abbreviate as \emph{$n$-\iet}

In Subsection~\ref{IETSANDFLIPS} we will analyse some properties of $n$-\iets\ with $n \in \N$, while the case of $\infty$-\iets \,is relegated to Section~\ref{infiniteiet}; beforehand, we present some definitions equally valid for both in Subsection~\ref{S:defs}.

\subsection{Definitions}\label{S:defs}
Let $T:D \to [a,b]$ be an $n$-\iet\ ($n \in \N$ or $n=\infty$) and $\{I_i=(a_i,a_{i+1})\}_{i \in \N_n}$ be the collection of the connected components of $D$. Observe that $T^{-1}$ is also an $n$-\iet\ The points in $\{a_i\}_{i \in \N_{n+1}}$ are called the \textit{discontinuities of $T$}. A discontinuity $a_i \notin \{a,b\}$ is said to be a \emph{false discontinuity} if $\lim_{x\to a_i^+} T(x)=\lim_{x\to a_i^-} T(x)$. In the absence of false discontinuities we say that $T$ is a \textit{proper} $n$-\iet, in whose case $T^{-1}$ is also proper. If $T$ \emph{reverses the orientation} in some of the interval $I_i$ (i. e. the slope is $-1$ in that interval) we say $T$ is an $n$-\iet\ with flips; otherwise we can say that $T$ is an \iet\ \emph{without flips} or an \emph{oriented} \iet\ When $T$ reverses the orientation exactly in $k$ components of $D$, we remark it by saying that $T$ is an interval exchange transformation of $n$-intervals with $k$-flips or, simply, \emph{an (n,k)-\iet} 

If we replace $[a,b]$ by $\S^1=[a,b]/\equiv$, (where $a \equiv b$), we receive the notion of \emph{circle exchange transformation of $n$-intervals}, abbreviated as \emph{$n$-\cet} Given an $n$-\iet\ $T:D \to [a,b]$ as above, we will denote as $T^c$ the $n$-\cet\ obtained after identifying $a$ and $b$. The meaning of the notions of $n$-\cets\ (respectively \gcets) with flips and of \emph{$(n,k)$-\cets} are obvious; same comment works for the concepts of false discontinuities and of properness. When working with $\S^1=[a,b]/\equiv$, and for the 
sake of simplicity, given any $x,y \in [a,b]$ we will still name them as $x$ and $y$ seen as points in $\S^1$ (with the only precaution that $a=b$ in $\S^1$). Given two points $x<y$ in $[a,b]$ (which are different when they are seen in $\S^1$), the set $\S^1 \setminus \{x,y\}$ posses two components, two open arcs: one of them is exactly the interval $(x,y) \subset [a,b]$ seen in $\S^1$ (under the convention above), the other one will be denoted as $(y,x)$ (this corresponds with the points $[a,x) \cup (y,b] \subset [a,b]$). 

\def\O{\mathcal{O}}

Let $T$ be an $n$-\iet\ with $n \in \N\cup \{\infty\}$(respectively an $n$-\cet). Let $x \in (a,b)$ 
 (respectively $x \in \S^1$) then the forward (respectively backward) orbit of $x$ generated by  $T$ is the set
$\O_T^+(x)=\{T^m(x): m\in\N\cup\{0\}  \textrm{  and } T^m(x) \textrm{ is defined} \}$
(respectively  $\O_T^-(x)=\{T^{-m}(x): m\in\N\cup\{0\}  \textrm{  and } T^{-m}(x) \textrm{ is defined} \}$).
The  \emph{orbit} of $x$ generated by  $T$ is $\O_T(x)=\O_T^-(x)\cup\O_T^+(x)$. Moreover, reducing in this sentence only to case of $T$ being an \iet, we define $\O_T(a)=\{a\}\cup\O_T(\lim_{x\to a^+}T(x))$ and 
$\O_T(b)=\{b\}\cup\O_T(\lim_{x\to b^-}T(x))$. 
We say that $T$ is \emph{minimal} (respectively \textit{transitive}) if for any $x \in [a,b]$
(respectively if for some $x \in [a,b]$), $\O_T(x)$ is dense in $[a,b]$; 
this implicitly means that, in particular, $x$ has
either a \emph{full forward orbit} ($T^n(x)$ is 
defined for any $n\geq 0$) or a \emph{full backward orbit} 
($T^n(x)$ is defined for any $n\leq
0$).  A point $x\in(a,b)$ is said to have \emph{full orbit} 
if it has full backward and forward orbit. 

\def\rr{
T(p)}
\def\+{\oplus}
\def\-{\ominus}
\def\X{\otimes}
\def\Card{\mathrm{Card}}
\def\orbit{\O_S(x)}

\subsection{Minimal interval exchange transformations}\label{IETSANDFLIPS}
For any pair $(n,k) \in \N^2$ with $1\leq k \leq n$ and $n + k \leq 4$, there 
are no minimal $(n,k)$-\iet\ (in fact there are no transitive $(n,k)$-\cet, 
as Gutierrez et al. proved in~\cite{gutierrez4b}). For all the rest of the pairs 
$(n,k)$ with $n \in \N$ and $1 \leq k \leq n$ it is always possible to consider 
a minimal $(n,k)$-\iet\ The role of this subsection is to clarify these claims.
 
Here and subsequently, when working with an $n$-\iet\ $T:D \to [a,b]$, for some $n \in \N$, with $D \subset (a,b)$ having as connected components the open intervals $\{I_i=(a_i,a_{i+1})\}_{1 \leq i \leq n}$ we will always assume that $a=a_1 < a_2 < \cdots <a_{n+1}=b$. 

We will write $T(a_i^\+)=\lim_{x\to a_i^+}T(x)$ for $1\leq i\leq n$  and $T(a_i^\-)=\lim_{x\to a_i^-}T(x)$ for $2\leq i\leq n+1$. We also write $T(a_1^\-)=T(a_1^\+)$ and $T(a_{n+1}^\+)=T(a_{n+1}^\-)$. A \emph{saddle  connection} for $T$ is a set $\mathcal{S}=\{a_i,T(a_i^{\X}),\dots ,T^k(a_i^{\X})=a_j\}$ with $k\geq 1$, $\X\in\{\+,\-\}$, $\mathcal{S}\cap\{a_r\}_{r=1}^{n+1}=\{a_i,a_j\}$ and possibly $i=j$. Observe that any $\iet$\ has saddle connections with $j\in\{1, n+1\}$ and $\Card (\mathcal{S})=1$ or $2$, these are called 
\emph{trivial saddle connections}. 

\begin{remark}\label{katok} When $T$ is minimal it is obvious that it has no nontrivial saddle connections (by definition of minimality we have, for every $2 \leq j \leq n$, $\mathcal{O}_T(a_j)=\mathcal{O}_T^{-}(a_j)$ is dense and therefore infinite in particular). Conversely, if $T$ has no nontrivial saddle connections, then $T$ is minimal and in fact any forward or backward orbit through any point is dense when it exists (see~\cite[Corollary~14.5.12]{katokhasselblatt}). It is important to stress that in the statement of \cite[Corollary 14.5.12]{katokhasselblatt} the hypothesis on the absence of saddle connections refers to the absence of nontrivial saddle connections.
\end{remark}

\def\codes{\mathcal{C}_n}

There is a natural
injection between the set of $n$-\iets\ and
$\codes=\Lambda^n\times S_n^\sigma$, where 
$\Lambda^n=(0,+ \infty)^n$ and $S_n^\sigma$ 
is the set of \emph{(signed) permutations}, where by a  permutation we mean an injective 
map, $\pi:\N_n=\{1,2,\dots,n\}\to \N_n^\sigma=\{-n,-(n-1),\dots,-1,1,2,\dots,n\}$, such that
 $|\pi|:\N_n\to \N_n$ is bijective. A signed permutation $\pi$ is said to be a \emph{non-standard 
permutation} if it is such that for some $i$ it holds $\pi(i)<0$ (otherwise, $\pi$ is simply a 
\emph{standard permutation}). As in the case of standard permutations,
$\pi$ will be represented by the ordered $n$-upla $(\pi(1),\pi(2),\dots,\pi(n))$. 

Let $T$ be an $n$-\iet\ like above,
then its associated coordinates in $\codes$ are $(\lambda,\pi)$ where $\lambda=(\lambda_i)_{i}$ with 
$\lambda_i=a_{i+1}-a_i$ for all $i\in \N_n$ and with $\pi(i)$ being positive (respectively negative) 
if $T|_{I_i}$ has slope $1$ (respectively $-1$) and such that, if we order the set $\{T(I_i)\}_{i=1}^n$
in a $n$-upla taking into account the usual order in $\R$, then $|\pi(i)|$ is the position of the
interval $T(I_i)$ in that $n$-upla.

Conversely, given any $(\lambda,\pi) \in \Lambda^n \times S_n^\sigma$ we can associate it an $n$-\iet, $T: \cup_{i=1}^{n} I_i \subset [0,b]\to[0,b]$, where $b=\sum_{i=1}^{n} \lambda_i$, $I_1=(0,\lambda_1)$, $I_i=(\sum_{j=1}^{i-1} \lambda_j,\sum_{j=1}^{i} \lambda_j)$ for any
$1<i\leq n$ and 
\begin{equation}\label{formulaietgen}
T|_{I_i}(x)= \left(\sum_{j=1}^{ \left|\pi\right|(i) - \frac{\sigma(\pi(i)) + 1}{2}}{\lambda_{\left|\pi\right|^{-1}(j)}}\right) + \sigma(\pi(i)) \left[x - \left(\sum_{j=1}^{i-1} \lambda_j\right)\right]
\end{equation}
for any $1 \leq i \leq n$, where $\sigma(z)=\frac{z}{\left|z\right|}$ (the sign of $z$). 

These coordinates allow us to make the identification
$T\equiv (\lambda,\pi)$.

\def\Z{\mathbb{Z}}

Notice that if $T \equiv (\lambda,\pi)$ then $T^{-1}\equiv (\mu,\tau)$ with $\tau(j)=\sigma(\pi(|\pi|^{-1}(j)))|\pi|^{-1}(j)$ and 
$\mu_j=\lambda_{|\pi|^{-1}}(j)$. Combining this fact with Equation~\eqref{formulaietgen} we see that,
for any $m\in\Z$, if $x$ is in the domain of $T^m$ then 
\begin{equation}\label{formulaiet} 
T^m(x)= \sigma(m) x + k_1(m) \lambda_1 + \cdots + k_n(m) \lambda_n
\end{equation}
with $\sigma(m) \in \{-1,1\}$ and for certain $k_1(m),\ldots, k_n(m) \in \Z$ (all depending on $x$).

\def\O{\mathcal{O}}

Minimal \iets\ and \cets\ without flips were characterized many years ago by M. Keane (see~\cite{keane}). Let $T$ be an $n$-\iet\ in $(a,b)$ without flips and with domain 
$D=\bigcup_{i=1}^{n} (a_i,a_{i+1})$. Let $\overline{T}$ be the right continuous extension of $T$ to $[a,b)$. Then, we say that $T$ satisfies the \textit{Keane condition} if $ \overline{T}^{m}(a_i) \neq a_j \text{ for all } m \geq 1, \, 1 \leq i, j \leq n \text{ and } j\neq 1.$

\begin{theorem}[Keane]\label{keaneth} Let $T$ be an oriented $n$-\iet\ that satisfies the Keane condition, then $T$ is minimal.
\end{theorem}

A permutation $\pi: \N_n \rightarrow \N_n^{\sigma}$ is called \textit{irreducible} if, for any $1 \leq l < n$, $\left|\pi(\{1,2,\ldots,l\})\right| \neq \{1,2,\ldots,l\}$. Naturally, if $T \equiv (\lambda,\pi)$ is a minimal \iet\, then $\pi$ must be irreducible. On the other hand, when $T$ has no flips and the components of $\lambda$ are rationally independent, the converse is also true. We formalize this statement in the following lemma (stated and proved in~\cite{keane}).

\begin{lemma}\label{orientedIETS} If $T \equiv (\lambda,\pi)$ is an \iet\ without flips, $\pi$ is irreducible and the components of $\lambda$ are rationally independent, $T$ satisfies the Keane condition and hence it is minimal. 
\end{lemma}

For the case of \iets\ with flips things changes substantially, see \cite{angostosoler} for a characterization of minimality. A. Linero and the third author have recently obtained~\cite{linsol} the following result,
which will play an essential role in this paper.

\begin{theorem}[Linero and Soler]\label{minlinsol} Let $n\geq 4$ and 
$1 \leq k \leq n$. Then there exists a  proper, 
minimal and uniquely ergodic $(n,k)$-\iet
\end{theorem}

\begin{remark} \label{R-minlinsol} 
A  proper minimal $(n,k)$-\iet\ in $(a,b)$, $T$,  always generate a minimal $(n,k)$-\cet, $T^c$,  after identifying $a$ and $b$. However the second does not have to be proper if we receive a false discontinuity in $a\equiv b$; this occurs if for some $1 \leq i \leq n$, either $T$ preserves the orientation in both $I_i$ and $I_{i+1}$, $\lim_{x \to {a_{i+1}}^-}{T(x)}=b$ and $\lim_{x \to {a_{i+1}}^+}{T(x)}=a$ (in this case $T^c$ is a proper minimal $(n-1,k)$-\cet) or $T$ reverses the orientation in both $I_i$ and $I_{i+1}$, $\lim_{x \to {a_{i+1}}^-}{T(x)}=a$ and $\lim_{x \to {a_{i+1}}^+}{T(x)}=b$ (here $T^c$ is a proper minimal $(n-1,k-1)$-\cet). Here, we are calling $I_{n+1} = I_1$. In other words, a proper minimal $(n,k)$-$\iet$, $T$, with coordinates $(\lambda, \pi)$, produces a proper minimal $(n-1,k)$-\cet\ (respectively a minimal $(n-1,k-1)$-\cet) only if, after calling $\pi(n+1)=\pi(1)$, we have $\pi(i)=n$ and $\pi(i+1)=1$  (respectively $\pi(i)=-1$ and $\pi(i+1)=-n$) for some $1
  \leq i \leq n$. So if $T$ is proper and we suppose $T^c$ has been extended by continuity, there exists $n' \in \{n-1,n\}$ and points $\{c_i\}_{i \in \N_{n' + 1}} \subset \{a_i\}_{i \in \N_{n +1}}$ such that $T^c$ is a proper $n'$-\cet\ which exchanges the intervals $(c_i,c_{i+1})$.

The \iet\ in $(a,b)$ given by the previous result in the case $k=n-2$ can in fact be taken with the extra property of obtaining a $(n-1,k)$-\cet\ after identifying $a$ and $b$ in $[a,b]$. Indeed, in~\cite{linsol} the authors build a minimal proper and uniquely ergodic $(n,n-2)$-\iet, $T\equiv(\lambda,\pi)$,  with $\pi=(-3,-4,-5,\dots,-[n-1],n,1,-2)$.
\end{remark}

\section{Infinite interval exchange transformations and proof of Proposition~\ref{P:mainiet}} \label{infiniteiet}
Despite the fact that there are some examples of \giets\ in the literature, see for example Chacon transformations in \cite{chacon} and the interesting way of modifying \giets\ analysed by Gutierrez et al. in~\cite{gutierrez4}, we have not found such examples when dealing with $\infty$-\iets\ with flips. We dedicate this section to fill this gap.

For the sake of clarity, we divide our exposition in two subsections. In Subsection~\ref{Sub1}, we present a procedure for building new minimal \iets\ modifying the definition of a given \iet\ in certain interval. In Subsection~\ref{Sub2}, we iterate that method to construct examples of $\infty$-\iets\ with flips and, in particular, to prove Proposition~\ref{P:mainiet}.

\subsection{Modifying minimal \iets\ in intervals}\label{Sub1}
Let us consider a proper $(n,k)$-\iet\ in $(0,1)$, $T:D=\cup_{i=1}^n (a_i,a_{i+1}) \to [0,1]$, for some $1 \leq k \leq n < \infty$ and an interval $\Delta \subset [0,1]$. 

Due to the \textit{Poincar\'e Recurrence Theorem}, the set $D_{\Delta}=\{x \in \Delta \, : \, T^m (x) \in  \Delta \text{ for some } m \in \N\}$ contains almost every point of $\Delta$ (i. e. $\Delta \setminus D_{\Delta}$ has zero Lebesgue measure). Associated with this $D_{\Delta}$, we may define the \textit{Poincar\'e map} (or the \textit{first return map}) \textit{of $T$ on $\Delta$} as the map $T_{\Delta}:D_\Delta \to \Delta$ which takes every point $x \in D_{\Delta}$ to $T_{\Delta}(x)=T^{m_x}(x)$ where $m_x = \min \{m \in \N \, : \, T^m(x) \in \Delta\}$. As it is proved in~\cite[Lemma~14.5.7]{katokhasselblatt},  $T_\Delta$ determines a proper $(n_\Delta,k_\Delta)$-\iet\ in the interval $\Inte(\Delta)$ with $1 \leq k_{\Delta} \leq n_{\Delta}\leq n+2$.

\begin{remark}\label{L:factsiets}  A precise examination of the proof of~\cite[Lemma~14.5.7]{katokhasselblatt} shows that if $c\in\Inte \Delta$ is a point of discontinuity of $T_\Delta$, then there exists $m\in \N$  such that $\{T^l(c)\}_{l=1}^{m-1}\cap \Delta=\emptyset$ and either $T^m(c)=a_j$ for some $j\in\{2,\dots,n\}$ or $T^m(c) \in \partial \Delta$. \end{remark}

Combining $T$ and $T_\Delta$ we are now able to consider a new \iet\ in $(0,1)$ with more discontinuities than $T$. 

\begin{definition} We call $D_\Delta^*= (D\setminus\Delta) \, \dot{\cup} \, D_\Delta$ and consider the map $T^*_\Delta:D_\Delta^*\to[0,1]$ given by
$$
T^*_\Delta ( x ) \, = \, \left\{
\begin{array}{lcl}
                   T(x) \, , && \textrm{ if }x\in D \setminus \Delta,\\                  
                   (T\circ T_\Delta)(x) \, , && \textrm{ if }x\in D_\Delta.\\                  
\end{array}
\right.
$$
\end{definition}

It follows directly from the definition, that $T^*_\Delta$ gives an \iet\ in $(0,1)$ which has as discontinuity set the union of the discontinuity sets of $T$, $\{a_i\}_{i=1}^{n + 1}$, and of $T_\Delta$, $\{c_i\}_{i=1}^{n_{\Delta} + 1}$. Also, for every $x \in [0,1]$, $\mathcal{O}_{T^*_\Delta}(x)\subset\mathcal{O}_{T}(x)$.



\def\Tacelerada{T^*_{\Delta}}

\def\condicionuno{\partial\Delta \cap \bigcup_{i=1}^{n+1}\left({\O_T(a_i)\cup\O_T(T(a_i^{\+}))\cup\O_T(T(a_i^{\-}))}\right)=\emptyset}
\def\condiciondos{\O_T(d)\cap\O_{T}(f)=\emptyset}
\def\nocondicionuno{\partial\Delta\cap\bigcup_{i=1}^{n+1}\left(\O_T(a_i)\cup\O_T(T(a_i^{\+}))\cup\O_T(T(a_i^{\-}))\right)\not=\emptyset}
\def\nocondiciondos{\O_T(d)\cap\O_{T}(f)\not=\emptyset}

\begin{lemma}\label{minBasic} If in the procedure above we suppose that $T$ is minimal, that $\Delta=(d,f)$ does not contain discontinuity points of $T$, that 
\begin{equation}\label{condicionuno}
 \condicionuno
\end{equation}
and that \begin{equation}\label{condiciondos}
\condiciondos,     \end{equation}
then $T_\Delta^*$ is a minimal $(n_\Delta^*,k_\Delta^*)$-iet with $n_\Delta^*=2n+3$ and $k_\Delta^* \geq 1$. \end{lemma}

\begin{proof}
According to Remark~\ref{katok}, showing the non existence of nontrivial saddle connections for $T_\Delta^*$ is sufficient to guaranteeing its minimality. Let us proceed by contradiction and let $\mathcal{S}$ be such a nontrivial saddle connection. There exists $k\in\N$ and 
$\X\in\{\-,\+\}$ such that $\mathcal{S}$ is in of one of the following four cases. 

\begin{description}
	\item[Case 1] $\mathcal{S}=\{a_i,\Tacelerada(a_i^{\X}),\dots,(\Tacelerada)^k(a_i^{\X})=a_j\} \text{ with } 1\leq i,j\leq n+1.$ In this case $\mathcal{S}$ is clearly contained in  a nontrivial saddle connection of $T$ which contradicts its minimality. 
	\item[Case 2] $\mathcal{S}=\{a_i,\Tacelerada(a_i^{\X}),\dots,(\Tacelerada)^k(a_i^{\X})=c_j\} \text{ with } 1\leq i\leq n+1 \text{ and } 1\leq j\leq n_\Delta+1.$ Observe that, in light of \eqref{condicionuno}, $j\not\in\{1,n_{\Delta}+1\}$. 
 Remark~\ref{L:factsiets} and again \eqref{condicionuno} imply the existence  of  $h\in \N \cup \{0\}$ and $2\leq l\leq n$ for which $T^{h}(c_j)=a_l$; on the other hand, the equality $(\Tacelerada)^k(a_i^{\X})=c_j$ means in particular that there exists $m\in\N $ such that 
 $T^m(a_i^{\X})=c_j$. Then we deduce that $\mathcal{S}'=\{a_i,T(a_i^{\X}),\dots,T^{h+m}(a_i^{\X})=a_l\}$ is a nontrivial saddle connection of $T$.
	\item[Case 3] $\mathcal{S}=\{c_i,\Tacelerada(c_i^{\X}),\dots,(\Tacelerada)^k(c_i^{\X})=a_j\} \text{ with }  1\leq i\leq n_\Delta+1 \text{ and } 1\leq j\leq n+1.$ As before, \eqref{condicionuno} and Remark~\ref{L:factsiets} guarantee that  $i\not\in\{1,n_{\Delta}+1\}$ and the existence of $h \in \N \cup \{0\}$, $m\in\N$, $2\leq l\leq n$, satisfying $T^h(c_i)=a_l$ and $T^m(a_l^ \oslash)=\Tacelerada(c_i^{\X})$ for some
 $\oslash\in\{\-,\+\}$. Thus we obtain the existence of  $p\in\N$ for which $T^p(a_l^\oslash)=a_j$ and 
 $\mathcal{S}'=\{a_l,T(a_i^{\oslash}),\dots,T^{p}(a_l^{\oslash})=a_j\}$ is a nontrivial saddle connection of $T$.
	\item[Case 4] $\mathcal{S}=\{c_i,\Tacelerada(c_i^{\X}),\dots,(\Tacelerada)^k(c_i^{\X})=c_j\} \text{ with } 1\leq i,j\leq n_\Delta+1.$  First, assume that $\{i,j\}\cap\{1,n_\Delta+1\}\not=\emptyset$ then several possibilities   arises. The first one,   $\{i,j\}=\{1,n_\Delta+1\}$, cannot occur since it contradicts \eqref{condiciondos}, then either $i\in\{1,n_\Delta+1\}$ and   $j\not\in\{1,n_\Delta+1\}$   or $j\in\{1,n_\Delta+1\}$ and $i\not\in\{1,n_\Delta+1\}$. In both cases, reasoning respectively as in the second and third item, we obtain a contradiction with \eqref{condicionuno}. Assume now that $2\leq i,j\leq n_\Delta$,  reasoning in the same manner that in the previous item  we can obtain $p\in\N$, $2\leq l\leq n$  and $\oslash\in\{\-,\+\}$ for which $T^p(a_l^\oslash)=c_i$. Now we obtain the existence of $m\in \N $ and  $2\leq s\leq n$ from  Remark~\ref{L:factsiets} such that $T^m(c_j)=a_s$. Then $T^{p+m}(a_l^\oslash)=a_s$ which implies again the existence of a nontrivial saddle 
 connection for $T$. 
\end{description}

Furthermore, since $T$ is minimal we have that for every $2 \leq i \leq n$ the backward orbit $\mathcal{O}_T^-{(a_i)}$ meets the open interval $\Delta$: this produces a minimum of $n-1$ discontinuity points of $T_{\Delta}$ in $\Delta$. Moreover, conditions \eqref{condicionuno} and \eqref{condiciondos} and the density of the backward orbits of $d$ and $f$ produce two more discontinuity points of $T_{\Delta}$ different from these $n-1$ previous ones. In total there are exactly $n+1$ discontinuity points of $T_{\Delta}$ in $\Inte(\Delta)$. Since $\Delta \subset (a_i,a_{i+1})$ for some $1 \leq i \leq n$, and by condition \eqref{condicionuno} we know that in fact it must be $[d,f] \subset I_i$, we can conclude that $n_\Delta^*=2 n + 3$.
\end{proof}

\subsection{Proof of Proposition~\ref{P:mainiet}}\label{Sub2}
Let us now begin with a proper minimal $(n,k)$-\iet\ in $(0,1)$, $T:D=\cup_{i=1}^n{(a_i,a_{i+1})} \to [0,1]$, a fix dense set $\{x_i\}_{i\in\N}$ on $[0,1]$ and a point $p\in (a_1,a_2)$ with full orbit.

We shall build inductively a sequence of \iets, $(S_i)_{i \in \N \cup \{0\}}$, whose combination allows us to get an example of a minimal $\infty$-\iet\ with flips. 

We start defining $S_0=T$. 

Given any $i\in\N$, suppose $S_{i-1}$ has been already defined with the property of being a minimal \iet\ such that $p$ has full orbit under it. Hence, by Remark~\ref{katok}, there must exist a minimal natural $n_{i}$ such that for every $1\leq h \leq i$ both $\{S_{i-1}^j(p)\}_{j=0}^{n_i}$ and $\{S_{i-1}^j(p)\}_{j=-n_i}^{0}$ meet $\left(x_h-\frac1i,x_h+\frac1i\right)$. We then define $S_{i}=(S_{i-1})^*_{\Delta_{i}}$ where $\Delta_i=(d_i,f_i) \subset (p,a_2)$ is such that
\renewcommand{\theenumi}{C-\arabic{enumi}}


\begin{enumerate}
 \item\label{E1} $|\Delta_i|=f_i-d_i<d_i-p$ and, when $i\geq 2$, $f_i<d_{i-1}$;
 \item\label{E2} $f_i - p < 1/i$;
\item\label{E3} $\partial \Delta_{i} \cap \O_{T}(p) = \emptyset$;  
 \item\label{E5} If $Q_{i-1}$ is the set of discontinuities of $S_{i-1}$, then
 \begin{equation*}
\partial \Delta_{i}\bigcap \bigcup_{x\in Q_{i-1}}(\O_{S_{i-1}}(x)\cup\O_{S_{i-1}}(S_{i-1}(x^\+))\cup\O_{S_{i-1}}(S_{i-1}(x^\-)) = \emptyset,   
 \end{equation*}
and
 \begin{equation*}
\O_{S_{i-1}}(d_{i})\cap\O_{S_{i-1}}(f_{i})=\emptyset;
 \end{equation*}
  \item\label{E6}  $\Delta_i \cap \{S_{i-1}^j(p)\}_{j=-n_i}^{n_i}=\emptyset$.
\end{enumerate}

Notice that, because of Lemma~\ref{minBasic}, properties~\ref{E1} and~\ref{E5} ensure that $S_{i}$ is also a minimal \iet\ and, because of property~\ref{E3}, $p$ has also full orbit for $S_i$. 

\renewcommand{\theenumi}{\arabic{enumi}}

Let us call $E=\bigcup_{i\in\N}\Delta_i$ and observe that, for every $i \in \N$,
 \begin{enumerate}
	\item $S_{i-1}|_{\Delta_i}$ is continuous;
  \item if $x \in (0,1) \setminus E$, then $S_i(x)=T(x)$;
  \item if $x \in \Delta_i$, then $S_k(x)=S_i(x)$ for any $k\geq i$.
 \end{enumerate}

This allows us to define $S:D_S \to (0,1)$, where $D_S=D \setminus \cup_{i\in\N}{Q_i}$,  by
$$
S(x)=\left\{\begin{array}{lcl}
             S_i(x)& &\textrm{ if }x\in\Delta_i,\\
             T(x)& &\textrm{ if }x\in D \setminus E.\\
            \end{array}
\right.$$

\begin{proposition}\label{T:ietinfinita} The function $S$ is a minimal $\infty$-\iet\ with flips.
\end{proposition}
\begin{proof} The fact that $S$ is an $\infty$-\iet\ with flips is clear. We prove the minimality of $S$ by stages.

Firstly, for any $x\in(0,1)$ the orbit $\O_S(x)$ is infinite, that is, either $x$ has full forward orbit or full backward orbit. Indeed, if for some $x\in(0,1)$ and some $\times\in\{-,+\}$ the set $\O_S^{\times}(x)$ is finite, we may take the maximal natural $i$ such that $\O_S^{\times}(x) \cap \Delta_i\not=\emptyset$ and observe that $\O_S^{\times}(x)=\O_{S_{i+1}}^{\times}(x)$. But, on account of the minimality of $S_{i+1}$, $\O_{S_{i+1}}^{-}(x)$ and $\O_{S_{i+1}}^{+}(x)$ cannot be simultaneously finite.
 
Secondly, $p\in\Cl (\orbit)$ for any $x\in(0,1)$. If  $p\not\in\Cl (\orbit)$, then  $\orbit$ would only intersect a finite number of intervals $\Delta_i$ (because of \ref{E2}) and, as in the previous paragraph, we arrive to a contradiction with the minimality of some  $S_k$, $k\in\N$.
 
Thirdly, $\lim_{x \to p^+}S(x)=T(p)$ and therefore $S$ is continuous on $p$. Indeed, fix $\epsilon>0$ and  observe that 
 $S(x)-T(x)=0$ if $x\in E^c$ and  $|S(x)-T(x)|<|\Delta_i|<x-p$ if $x\in\Delta_i$  (because of~\ref{E1}). Thus, for a sufficiently small $\delta>0$ we have that for every $x\in (p-\delta,p+\delta)$, $|S(x)-T(p)| \leq |S(x)-T(x)|+|T(x)-\rr| <2 |x-p| <\epsilon$.
 
Fourthly, both the backward and the forward orbits of $p$ generated by $S$ are dense in $[0,1]$. To prove it, it is clearly sufficient to see that, for any $h,i\in\N$ with $h \leq i$ and any $\times\in\{-,+\}$, $\O_S^\times(p)\cap\left(x_h-\frac1i,x_h+\frac1i\right)\not=\emptyset$.  
But if $h\leq i$ and $\times\in\{-,+\}$, we know that $\left\{S_{i-1}^j(p)\right\}_{j=0}^{\times n_i}\cap \left(x_h-\frac1i,x_h+\frac1i\right)\not=\emptyset$ and that, by \ref{E6}, $\left\{S^j(p)\right\}_{j=0}^{\times n_i}=\left\{S_{i-1}^j(p)\right\}_{j=0}^{\times n_i}$. 

 Finally, we take $x\in(0,1)$ and we see that for any $c\in(0,1)$ and any $\epsilon>0$, $\O_S(x)\cap (c-\epsilon,c+\epsilon)\not=\emptyset$.  Since $\O_S^\times(p)$ is dense (for any $\times\in\{-,+\}$) we can take an integer $m_2$ 
 (note that the sign can be chosen as desired) for which  
 $|S^{m_2}(p)-c|<\frac\epsilon2$. 
Observe that, because of \ref{E3} and the fact that $S$ is continuous at $p$, $S^{m_2}$ is continuous at $p$. So there exists $\delta>0$ such that if $|y-p|<\delta$ then 
 $|S^{m_2}(y)-S^{m_2}(p)|<\frac\epsilon2$. Since $p\in \Cl(\O_S(x))$, there exists also an integer $m_1$ satisfying 
 $|S^{m_1}(x)-p|<\delta$. Thus $|S^{m_1+m_2}(x)-c|\leq|S^{m_1+m_2}(x)-S^{m_2}(p)|+|S^{m_2}(p)-c|<\frac\epsilon2+\frac\epsilon2=\epsilon.$
 Then $S^{m_1+m_2}(x)\in(c-\epsilon,c+\epsilon)$ as desired.
\end{proof}

\section{Building vector fields from circle exchange transformations}\label{flowgiet}

Circle exchange transformations and vector fields are
related by means of a standard procedure called \emph{suspension of circle
exchange transformations}. In this section we introduce this procedure
following \cite[Section~6]{gutierrez4} (with minor changes). We have been also strongly inspired by 
\cite{arnoux,levitt} and \cite{capitulosoler}.

Let $D$ be an open dense subset of $(0,1)$ and $\{I_i=(a_i,a_{i+1})\}_{i \in \N_n}$ be the family of its connected components ($k \in \N$, $n \in \N \cup \{\infty\}$). Let $T: D \to [0,1]$ be a proper $(n,k)$-\iet\ (with $n \in \N \cup \{\infty\}$). When $n \in \N$, recall the convention of supposing that $0=a_1<a_2 < \cdots < a_{n+1}=1$. Take also the proper $(n',k')$-\cet\ $T^c:\cup_{i \in \N_{n'}} (c_i,c_{i+1}) \subset \S^1 \to \S^1$ associated to $T$ where $k'\in \{k-1,k\}$ and $n'\in \{n-1,n\}$ (see Remark~\ref{R-minlinsol}). Consider the set of points $\cup_{i \in \N_{n'}} \, \partial \, T^c((c_i,c_{i+1}))$ and label them as $\{b_i\}_{i \in \N_{n' + 1}} \subset [0,1]$ such as to verify, for every $i,j \in \N_{n' + 1}$, $c_i \leq c_j$ if and only if $b_i \leq b_j$. Let $\sigma: \N_{n'} \to \N_{n'}$ be a bijection such that, for every $i \in \N_{n'}$, $T^c((c_i,c_{i+1}))=(b_{\sigma(i)},b_{\sigma(i)+1})$.

\subsection{The construction of the suspended surface.}\label{SS:suspension}
Figure~\ref{F:Suspension} intends to  clarify the following construction. We start considering the noncompact $\partial$-surface $N=(\S^1\times[0,1]) \setminus \left(P_0 \cup P_1 \right)$ with $P_0 = \{b_i\}_{i \in \N_{n'}} \times \{0\}$ and $P_1 = \{c_i\}_{i \in \N_{n'}} \times \{1\}$. 

Call also $\bar{N}=\S^1\times[0,1]$ and consider both $N$ and $\bar{N}$ as $\partial$-surfaces equipped with the natural analytic structure compatible with their euclidean topological structure (as subsets of $\S^1\times\R$).

For every $i \in \N_{n'}$, we consider $h_i : [c_i,c_{i+1}] \times \{ 1\} \to [b_{\sigma(i)},b_{\sigma(i)+1}] \times \{ 0 \}$ given by the formula $h_i(x,1) = (T^c(x), 0)$ for every $x \in (c_i,c_{i+1})$ and either $h_i(c_i,1)=(b_{\sigma(i)},0)$ and 
$h_i(c_{i+1},1)=(b_{\sigma(i)+1},0)$ (if $T^c$ preserves the orientation in $(c_i,c_{i+1})$) or $h_i(c_{i+1},1)=(b_{\sigma(i)},0)$ and 
$h_i(c_{i},1)=(b_{\sigma(i)+1},0)$ (if $T^c$ reverses the orientation in $(c_i,c_{i+1})$).

Let $\mathcal{R} \subset \bar{N} \times \bar{N}$ be the smallest equivalence relation of $\bar{N}$ containing all the pairs $(p,q)$ satisfying that, for some $i \in \N_{n'}$, $p \in [c_i,c_{i+1}] \times \{ 1\}$ and $q = h_i(p)$. Consider the topological quotient $M_T=\bar{N}/\mathcal{R}$ and let $\rho: \bar{N} \rightarrow M_T$ be the associated natural projection. When $(x,y) \in \mathcal{R}$ we will also write $x \sim y$ or $\left[x\right]=\left[y\right]$. 


Since $\bar{N}$ is compact and connected so is $M_T$. The subset of $M_T$ given by $S_T=\rho(N)$ has not only topological structure but  it is also an analytic surface with boundary: $S_T$ is nothing else than a set built following the standard process of attaching surfaces along their boundaries. We collect this information in the following result.

\begin{figure}[ht]
	\centering
			\includegraphics[width=0.80\textwidth]{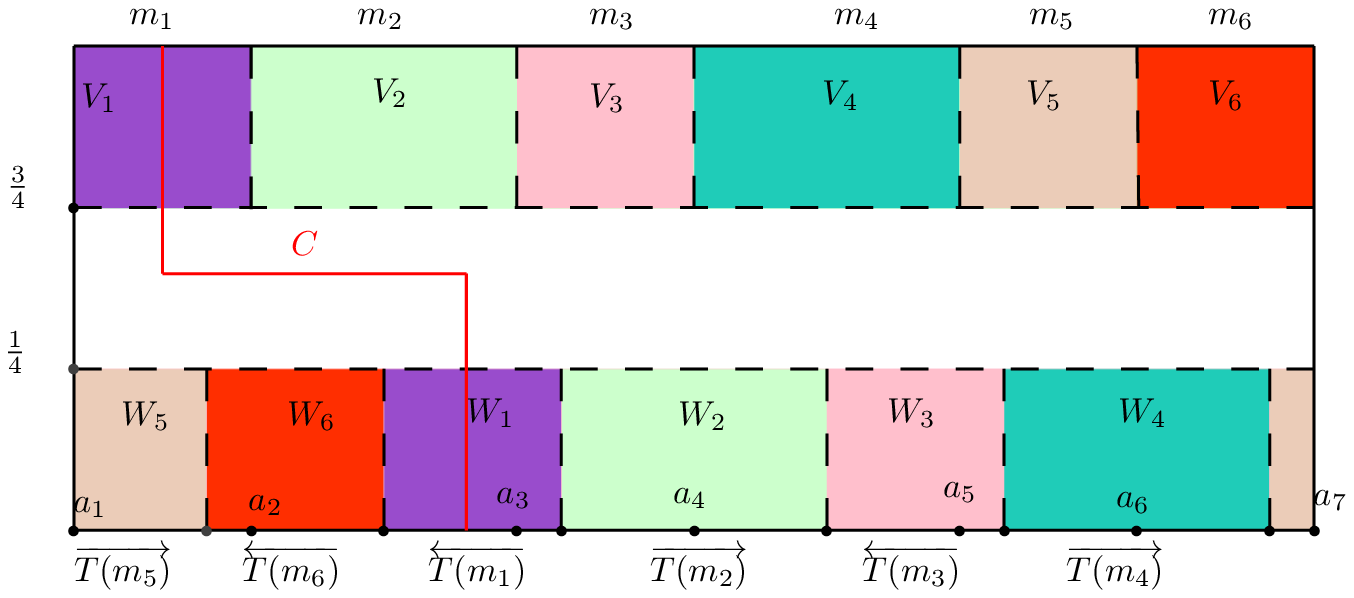} 
		\caption{Construction of $M_T$ by means of a $(6,3)$-\iet\ with $\pi=(-3,4,-5,6,1,-2)$. The circle $C$ is nonorientable. The arrows on the images of the $m_i$ mark if they are flipped by $T$}
	\label{F:Suspension}
\end{figure}

\begin{lemma}\label{surfacestructure} $S_T$ is an analytic surface. The surface is orientable (respectively nonorientable) when $T^c$ has no flips (respectively has flips).. Moreover, if $n \in \N$, $M_T$ is a compact surface which coincides with $S_T$ in genus and in orientability class.
\end{lemma} 
\begin{proof}
We start defining, for every index $i \in \N_{n'}$, two open subsets of $N$ as $V_{i} = (c_i,c_{i+1}) \times (3/4,1]$ and $W_{i} = (b_{\sigma(i)},b_{\sigma(i)+1}) \times [0,1/4)$ and a continuous maps $\phi_{i} : V_{i} \cup W_{i} \rightarrow (b_{\sigma(i)},b_{\sigma(i)+1}) \times (-1/4,1/4)$ by
\begin{equation}\label{chart1} \phi_{i}(x,t)= \begin{cases} (T^c(x),t - 1), \, \,  \text{ if } (x,t) \in V_{i}, \\ 
                                (x, t), \, \, \text{ if } (x,t) \in W_{i}.
\end{cases} \end{equation}

Since the restrictions of $\phi_{i}$ to $V_{i}$ and to $W_{i}$ are both embeddings with closed images in $(b_{\sigma(i)},b_{\sigma(i)+1}) \times (-1/4,1/4)$, it follows that $\phi_{i}$ is not only continuous but also closed. Moreover, since $\phi_{i}$ 
 take the same value in all points $(x,y)$ and $(x',y')$ such that $\rho(x,y)=\rho(x',y')$, 
we can define a map $\Phi_{i}: \rho(V_{i} \cup W_{i}) \rightarrow (b_{\sigma(i)},b_{\sigma(i)+1}) \times (-1/4,1/4)$ by $\Phi_{i}(\rho(x,t))=\phi_{i}(x,t)$ for every $(x,t) \in V_{i} \cup W_{i}$. This map is bijective and the continuity and closeness of $\phi_{i,j}$ guarantees that it is, in fact, an homeomorphism and, in particular, that $\rho(V_{i} \cup W_{i})$ is a surface. On the other hand, the restriction of $\rho$ to $\Inte N$ is an embedding and, therefore, every point in $S_T$ posses an open neighbourhood homeomorphic to an open connected subset of $\mathbb{R}^2$. 

The fact that $S_T$ can be rewritten as a countable union of open subsets each of them being second countable shows that $S_T$ is second countable itself: $S_T= \rho(\Inte N) \cup \bigcup_{i \in \N_{n'}}{\rho(V_{i} \cup W_{i})}$. This union also allows us to claim the connectedness of $S_T$: all the sets in the union of the left term in the equality are connected and all of them meet $\rho(\Inte N)$. Finally, $S_T$ is clearly Hausdorff. All together shows that $S_T$ is a surface.

Let us denote by $(x,y)$ the components of the identity map on $\Inte N$ and by $(x^i,y^i)$ the components of the map $\Phi_i$ for every $i \in \N_{n'}$. As an atlas for $S_T$ we can take the collection of the following coordinates charts
$ \{(\rho(\Inte N), (x,y))\}\cup\{(\rho(V_{i} \cup W_{i}), (x^{i}, y^{i}))\}_{i \in \N_{n'}}$. It is an immediate computation to verify that the transition maps associated to these coordinate charts are all given by analytic maps and, when $T$ has no flips, have all positive Jacobian determinant everywhere in their domains. Indeed, the only coordinate neighbourhoods with non-empty intersections are the pairs $\{\rho(\Inte N), \rho(V_{i} \cup W_{i})\}_i$. Let $i \in \N_{n'}$, we have $\rho(\Inte N) \cap \rho(V_{i} \cup W_{i})=\rho (\Inte V_{i} \cup \Inte W_{i})$. Let $F_{i}$ be the restriction of $\phi_{i}=\Phi_{i} \circ \rho$ to $\Inte V_{i}$ in its domain and to $\phi_{i}(\Inte V_{i})$ in its codomain and with formula $F_{i}(x,y)=\phi_{i}(x,t)$ where the last term is computed as in equation~\eqref{chart1}. The transition map from $(\rho(\Inte N), (x,y))$ to $(\rho(V_{i} \cup W_{i}), (x^i,y^i))$ is the map $H_i:\Inte V_{i} \cup \Inte W_i \to \phi_{i}(\Inte V_{i} \cup \Inte W_i)$ defi
 ned as $H_i(x,y)=F_i(x,y)$ if $(x,y) \in \Inte V_i$ and as $H_i(x,y)=\phi_i(x,y)=(x,t)$ otherwise. All these transition maps are therefore analytic diffeomorphisms. Moreover, their Jacobian matrices are trivial. When $T$ preserves the orientation in $(c_i,c_{i+1})$, the map $F_{i}$ has as Jacobian matrix at any point in its domain the identity; when $T$ reverses the orientation in $(c_i,c_{i+1})$ the Jacobian matrix of $F_{i}$ at any point $(x,t)$ equals $ \left( \begin{smallmatrix} -1& 0\\ 0& 1 \end{smallmatrix} \right)$. 
From here we can already conclude that $S_T$ is an orientable surface when $T$ has no flips (in such a case all the transition maps have the the identity as Jacobian matrix so in particular all have positive determinant and hence we have an analytic and consistently oriented atlas for $S_T$). On the other hand, when $T$ has flips, say for example its slope equals $-1$ in $(c_i,c_{i+1})$ and take an open $\Gamma$ arc in $\Inte N$ having as endpoints the middle points of $(c_i,c_{i+1})\times \{1\}$ and $(b_{\sigma(i)},b_{\sigma(i)+1}) \times\{0\}$, then $\rho(\Cl(\Gamma))$ is clearly a nonorientable circle, see the circle $C$ in Figure~\ref{F:Suspension}.

Suppose now that we are in the case $n \in \N$. To begin with, a direct examination on the family of open subsets of $M_T$ show that $M_T$ is not only compact and connected but also Hausdorff. Besides, the set $M_T \setminus S_T$ is totally disconnected and nonseparating on $M_T$. Corollary~\ref{finitegenus} together with the Remark~\ref{exfinite} finish the proof.\end{proof}

In what follows, we shall refer to $S_T$ as the \textit{suspended surface associated to $T$}. When $T$ is an $n$-\iet for some $n \in \N$, $M_T$ will be called the \textit{compact suspended surface associated to $T$}. In this last case, the points in the set $\mathcal{M}_T=\{\left[(c_i,1)\right] : 1 \leq i \leq n' \}$ are called the \textit{marked points of $M_T$}.

\subsection{The construction of the suspended vector field}\label{S:suspendedflow}
With the notation introduced in the proof of Lemma~\ref{surfacestructure}, we can consider a vector field $X$ (respectively a $2$-form $\theta$) on $S_T$ defined in local coordinates as 
$$ X_p=\left. \frac{\partial}{\partial y}\right|_p \, \, \text{(respectively } \theta_p= d x _p \wedge d y _p\text{)}, \, \, \, \, \, \text{ if } p \in \rho(\Inte N) $$
and $$ X_p=\left. \frac{\partial}{\partial y^{i}}\right|_p \,  \text{(respectively } \theta_p= d x^i _p \wedge d y^i _p\text{)}, \, \, \, \, \, \text{ if } p \in \rho(V_{i} \cup W_{i}).
$$


The trivial shapes which have all the Jacobian matrices associated to the transition maps guarantee that these definitions are consistent (they agree in the non-disjoint coordinate neighbourhoods). Furthermore, $X$ and $\theta$ are analytic in the whole $S_T$ because so are its local representatives.

As a direct computation reveals, the interior product of $\theta$ with $X$ equals $\iota_X(\theta)=\alpha$ where $\alpha$ is the $1$-form which in local coordinates is given by $\alpha_p=-d x _p$ if $p \in \rho(\Inte N)$ and by $\alpha_p=- d x^i_p$ if $p \in \rho(V_{i} \cup W_{i})$. Thus, $X$ preserves the area associated to $\theta$.

Now, observe that $[\S^1 \times \{1/2\}]$ is a transverse circle to any orbit of $X$ and the Poincar\'e first return map which is defined over it is given exactly by $T^c$ (when we see $T^c$ as a map from $\S^1 \times \{1/2\}$ to $\S^1 \times \{1/2\}$ after naturally identified $\S^1 \times \{1/2\}$ with $\S^1$). This simple observation allow us to claim that $X$ is an analytic minimal vector field if and only if $T^c$ is minimal.


Associated to $X$ we can always take an analytic positive map $f_T:S_T \to \R$ such that $X_T= f_T X$ is a complete analytic vector field which preserves the area associated to the $2$-form $\theta_T= 1/f_T \theta$. Notice that the factor $f_T$ can be taken so that the area 2-form $\theta_T$ is complete (i.e. each end of the noncompact surface $S_T$ has infinite area). We shall refer to this analytic vector field $X_T$ (respectively to this $2$-form $\theta_T$) as \textit{the suspended vector field (respectively $2$-form) associated to $T$}.

\section{Proofs of Theorems~\ref{T:maingenerofinito} and~\ref{T:maingeneroinfinito}}\label{proofmainth}

Let us start with the proof of Theorem~\ref{T:maingeneroinfinito}, which is immediate from Theorem~\ref{T:ietinfinita}. Indeed, the latter theorem gives an infinite minimal \giet\ $T$ with flips, and associated to it we may take the suspended surface $S_T$, the suspended vector field $X_T$ and the suspended area $2$-form $\theta_T$. To conclude, we only need to notice that the surface $S_T$ must be of infinite genus due to the \textit{Structure Theorem} of~\cite{gutierrez1}.

In order to deal with the proof of Theorem~\ref{T:maingenerofinito}, we shall restrict now to the case $n \in \N$ in the procedure above. Recall that, since $n \in \N$, the suspended surface $S_T$ is a noncompact surfaced contained in the suspended compact surface $M_T$. The analytic map $f_T: S_T \to \R$ considered above to define the suspended vector field $X_T$ can in fact be taken such that after defining it as zero in the marked points $\mathcal{M}_T=M_T \setminus S_T$ we get a $C^{\infty}$ map on $M_T$. Then, the vector field $X_T$ can also be seen as a $C^{\infty}$ complete vector field on the compact surface $M_T$ whose restriction to $S_T$ is analytic. When $X$ is understood as a vector field on $M_T$ we will refer to it as the \textit{suspended compact vector field associated to $T$}. 

Given $p\in M_T$, denote by $\gamma_p: \R \to M_T$ the integral curve of $X_T$ starting at $p$. We already know that for every point $p\in S_T$, the orbit $\Gamma_p=\gamma_p(\R)$ is dense in $M_T$. Even more, owing to Remark~\ref{katok}, if $p \in [\mathcal{S}^1 \times \{1/2\}]$ there are three possibilities:
\begin{enumerate}
	\item if $p \notin [\{c_1, \ldots, c_{n'}\} \cup \{b_1, \ldots, b_{n'}\} \times \{1/2\}]$, then both 
	$\Gamma_p^+=\{\gamma_p(t) : t \in [0,\infty)\}$ and $\Gamma_p^-=\{\gamma_p(t) : t \in (-\infty,0]\}$ are dense in $M_T$;
	\item if $p \in [\{c_1, \ldots, c_{n'}\}\times \{1/2\}]$, $\Gamma_p^+=[(\{c_i\}\times [1/2,1)]$ with 
	$\lim_{t \to \infty}{\gamma_p(t)}=[(c_i,1)]$ and $\Gamma_p^-=\{\gamma_p(t) : t \in (-\infty,0]\}$ is dense;
	\item if $p \in [\{b_1, \ldots, b_{n'}\}\times \{1/2\}]$, $\Gamma_p^-=[(\{b_i\}\times (0,1/2])]$ with $\lim_{t \to -\infty}{\gamma_p(t)}=[(b_i,0)]$ and $\Gamma_p^+=\{\gamma_p(t) : t \in [0,\infty)\}$ is dense.
\end{enumerate}


\subsection{Computation of the genus of the surface $M_T$}

In order to deduce the genus of the surface $M_T$, we begin computing the index of the singular points of $X_T$. For doing so, it is enough to notice that all the singular points of $X_T$ have neighbourhoods which are topologically equivalent with an open disk on the plane decomposed in evenly many hyperbolic sectors. Let us formalize what we want to mean by this (c.f. \cite[pp. 17--18]{dumortierllibreartes}). 

Consider the euclidean planar vector field $Z(x,y)=(x,-y)$. Take also the set $S_h=\{(x,y) \in \R^2 : 0 \leq x, \, y \leq 1 \, \wedge \, xy < 1/2\}.$ Let $M$ be a surface and $p \in M$ be a singular point for a $C^1$ complete vector field $Y$ over $S$. Let $U$ be an open disk containing $p$. Suppose that $p$ is the only singular point of $Y$ in $U$ and that, for some $m \geq 1$, there are $2 m$ integral curves $\gamma_1, \ldots, \gamma_{2 m}: \R \to M$ whose images meet $S \setminus U$ and such that, for every $1\leq j \leq 2m$, either $\lim_{t \to -\infty} \gamma_j(t)=p$ or $\lim_{t \to \infty} \gamma_j(t)=p$. The $2 m$ orbits given by these integral curves divide $U$ in $2 m$ sets, $U_1, \ldots, U_{2 m}$, which are closed connected subsets of $U$ with pairwise disjoint interiors and each of them with the union of $\{p\}$ and two semi-orbits of these $2 m$ orbits as their frontiers (in $U$). We say that $U_1 \cup \cdots U_{2 m}$ is a decomposition of $U$ in $2 m$ \emph{hyperbolic se
 ctors} if for every $U_i$ there exists an homeomorphism $h : U_i \rightarrow S_h$ whose restriction to $\Inte{U_i}$ takes orbits of the restriction of $Y$ to $\Inte{U_i}$ onto orbits of the restriction of $Z$ to $\Inte(S_h)$ preserving their natural orientations as images of integral curves. In such a case we also say that $p$ is a \textit{$2 m$-saddle point} for the vector field $Y$. See Figure~\ref{GeneralizeSaddle}.

\begin{figure}[ht]
	\centering
	\begin{tabular}{ccc}
		\includegraphics[width=.3\textwidth]{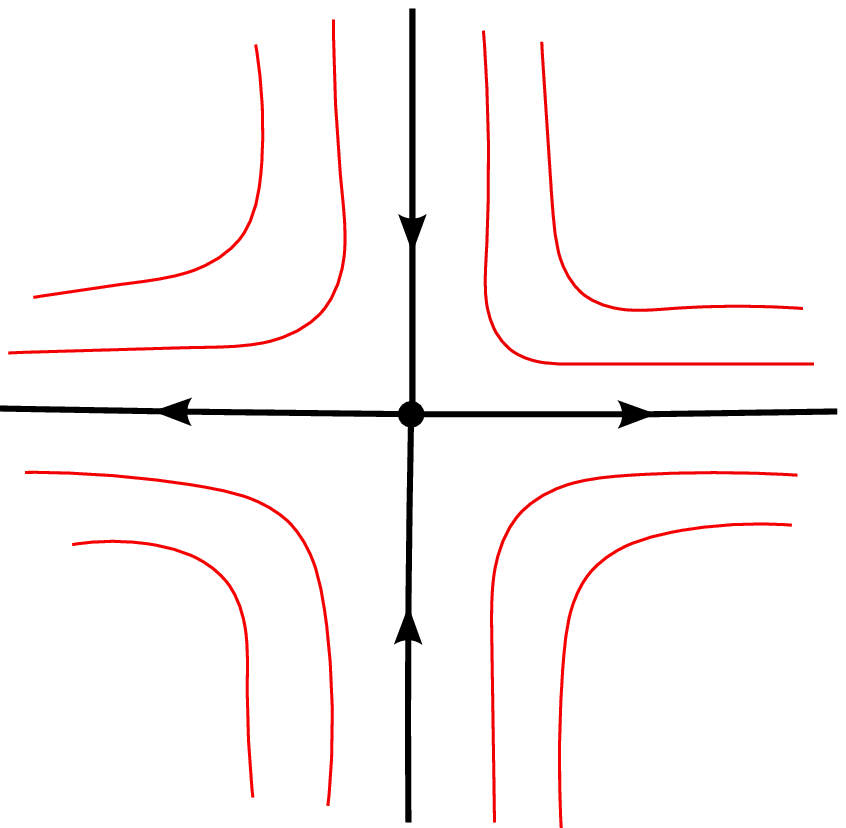}&\quad&
		\includegraphics[width=.3\textwidth]{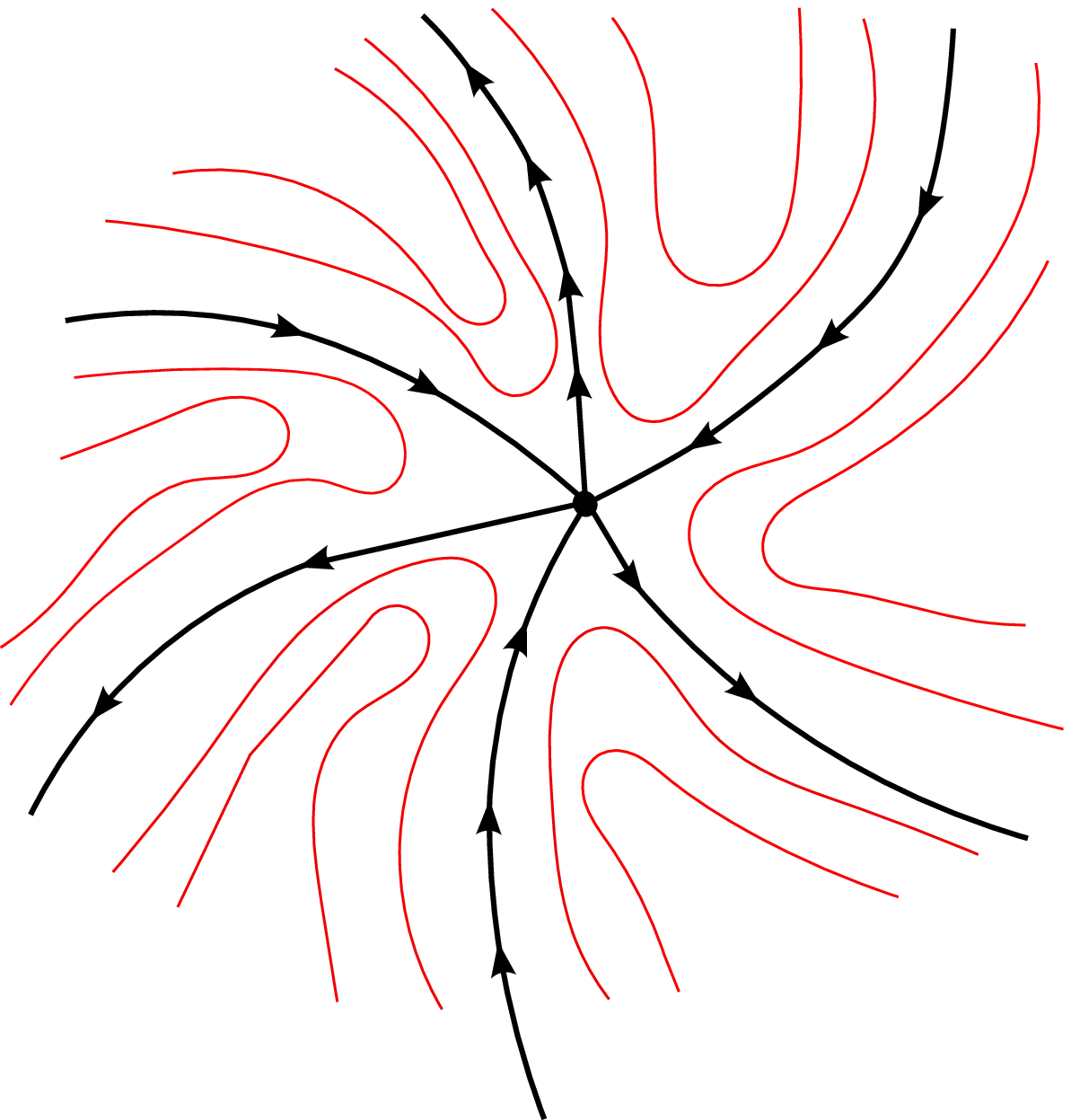}	 
	\end{tabular}
		\caption{A standard saddle point (left) and a $6$-saddle point (right)}
	\label{GeneralizeSaddle}
\end{figure}

Any of the marked points $[(c_i,1)]$, $1 \leq i \leq n'$, is a $2 i_k$-saddle point for $X_T$
for some $i_k \geq 2$. Indeed, since $T^c$ is proper, for any $1 \leq i \leq n'$ we have that $\lim_{x \to c_i^+}T^c(x) \neq \lim_{x \to c_i^-}T^c(x)$ and $\lim_{x \to b_i^+}(T^c)^{-1}(x) \neq \lim_{x \to b_i^-}(T^c)^{-1}(x)$ so the class of equivalence $[(c_i,1)]$ contains at least some other $[(c_j,1)]$ ($1 \leq j \leq n'$) with $c_j \neq c_i$ and as many points of the type $[(c_k,1)]$ as of the type $[(b_l,1)]$ ($1 \leq k,l \leq n'$). Let us say $[(c_i,1)]=\{(c_{i_1},1), \cdots, (c_{i_k},1)\} \cup \{(b_{j_1},0), \cdots, (b_{j_k},0)\}$ (with both $\{{i_1}, \cdots, {i_j}\}$ and $\{{j_1}, \cdots, {j_k}\}$ being sets with $k \geq 2$ different points in $\N_{n'}$). For every $1 \leq l \leq k$, the semi-orbit $\Gamma_{[(c_{j_l},1/2)]}^+$ has $\{[(c_i,1)]\}$ as $\omega$-limit set while the semi-orbit 
$\Gamma_{[(b_{j_l},1/2)]}^-$ has $\{[(c_{i},1)]\}$ as $\alpha$-limit set. If $B_i$ is a sufficiently small open ball centred in $[(c_{i},1)]$, then it is clear that the orbits of $X_T$ through any point of the semiorbit $\Gamma_{[(c_{j_l},1/2)]}^+$ or of the semiorbit $\Gamma_{[(b_{j_l},1/2)]}^-$, for $1 \leq l \leq k$, are the only regular ones meeting $B_i$ and having $[(c_{i},1)]$ in one of their limit sets. The rest of the regular orbits $\Gamma$ meeting $B_i$ are such that $\Gamma \cap B_i$ is an open arc with endpoints in the frontier of $B_i$. This produces exactly $2 i_k$ hyperbolic sectors in the decomposition of $B_i$ (\cite[Theorem~1.43, p. 35]{dumortierllibreartes} justify this geometrically clear claim). The index of the point $[(c_{i},1)]$ is then exactly $\frac{1}{2}(2 - 2 i_k)=1 - i_k$. 

By the Poincar\'e-Hopf Index Theorem, we can already compute the Euler 
characteristic of the surface $S_T$. Let us say that $M_T$ presents $m \geq 1$ marked points, then $ \chi(S_T)=\sum_{j=1}^m{(1 - i_k)}=m - n$. Consequently, if $S_T$ is orientable 
(respectively nonorientable) its genus equals 
\begin{equation}\label{E-genus}
g(S_T)=1 + \frac{n-m}{2}\; (\textrm{respectively } g(S_T)=2 + n - m).   
\end{equation}


\subsection{An intermediate result}

In this subsection we prove a special case of Theorem~\ref{T:maingenerofinito}, which plays a key role in the next subsection when proving the theorem in its full generality.
 
\begin{theorem}\label{UnPunto} For any $n \geq 4$ there exists a proper minimal $(n,n-2)$-\iet, $T$, which produces, after identifying $0$ and $1$, a $(n-1,n-2)$-\cet, $T^c$. The associated suspended compact surface, $M_{T}$, is a nonorientable compact surface of genus $n$ and the suspended compact vector field, $X_{T}$, is of class $C^{\infty}$ and has an only singular point $p_0$ (with a neighbourhood decomposed in $h=2 n - 2$ hyperbolic sectors). The restriction of $X_{T}$ to the suspended surface $S_T= M_T \setminus \{p_0\}$ is analytic and minimal. With more detail, given any point in $S_{T}$ if its orbit is not both backwardly and forwardly dense, then  either it has $\{p_0\}$ as $\alpha$-limit set and is forwardly dense or has $\{p_0\}$ as $\omega$-limit set and is backwardly dense (these last two cases arising only for exactly $h$ orbits).
\end{theorem}
\begin{proof} 
Theorem~\ref{minlinsol} and Remark~\ref{R-minlinsol} provide, for $n\geq 4$, a proper, minimal uniquely ergodic $(n,n-2)$-\iet, $T=(\lambda,\pi)$, with $\pi=(-3,-4,-5,\dots,-[n-1],n,1,-2)$. Denote by $(a_i,a_{i+1})$, $1\leq i\leq n$, the intervals exchanged by $T$. Take the minimal $(n-1,n-2)$-\cet, $T^c$, obtained after identifying $0$ with $1$ and write  
$(c_i, c_{i+1})$, $1\leq i\leq n-1$, to denote the intervals exchanged by $T^c$. Then we have $c_1=a_1<c_2=a_2<\dots<c_{n-2}=a_{n-2}<a_{n-1}<c_{n-1}=a_n<c_n=a_{n+1}=1$.

Next we show that the number of marked points appearing in the compact surface $M_T$ in the construction of the suspension is exactly 
one (see Section~\ref{SS:suspension}). By the construction of $M_T$ it is clear that 
the permutation $\pi$ gives the identifications $(c_i,1)\sim (c_{i+2},1)$ for any $1\leq i\leq n-4$. Furthermore, $\pi$ also gives the relations $(c_{n-3},1)\sim (c_{n-2},1)$, $(c_{n-1},1)\sim (c_{1},1)$ and  $(c_{2},1)\sim (c_{n-1},1)$.

Finally, from Equation~(\ref{E-genus}), we deduce that $g(S_{T^c})=n$.
\end{proof}

\begin{remark}\label{casoOrientable} The same argument can be used to guarantee that, considering appropriate $\cet$ without flips, the statement of the result works analogously for compact orientable surfaces of genus 
$g \geq 1$. Indeed, for any $n \in \N \cup \{0\}$ we may take the (standard) irreducible permutation $\pi$ given by $\pi(i)=2 i$ for every $1 \leq i \leq n + 1$ and $\pi(i)=2 (i - n + 2) +1$ for every $n+2 \leq i \leq 2 n + 2$: $  \pi=(2, 4, \ldots, 2 n, 2n+2, 1, 3, \ldots, 2 n - 1, 2 n + 1)$.

Consider also a vector $\lambda=(\lambda_i)_{1 \leq i \leq 2n +2}$ with its components being rationally independent. Then $T\equiv (\lambda, \pi)$ satisfies the Keane condition and it is minimal (see Lemma~\ref{orientedIETS}). After identifying $0$ and $1$ we get a proper minimal $(2n +1)$-\cet, $T^c$. This oriented \cet\ produces also a unique boundary component when we suspends it and equation~\eqref{E-genus} says now that $g(S_{T^c})=1 + \frac{(2n + 1) - 1}{2}=n+1.$
\end{remark}

\subsection{Completing the proof of Theorem~\ref{T:maingenerofinito}}
	
Using the previous results, we can now present the proof of Theorem~\ref{T:maingenerofinito}. Let $S$ be a nonorientable (respectively orientable) noncompact surface of finite genus $g\geq 4$  (respectively $g\geq 1$). According to Corollary~\ref{finitegenus}, given any compact nonorientable surface $S'$ of genus $g \geq 4$ (respectively any compact orientable surface of genus $g \geq 1$), there exists a (metric compact) totally disconnected subset $K \subset S'$ such that $S' \setminus K$ is homeomorphic to $S$ (and therefore analytic diffeomorphic, see~\cite[Theorem~2.1]{victorgabi}).
 
Let $n=g$ (respectively $n=g-1$), take $T$ a $(n,n-2)$-\iet\ (respectively a $2n +2$-\iet) as in Theorem~\ref{UnPunto} (respectively as in Remark~\ref{casoOrientable}) and let $T^c$ be the associated $(n-1,n-2)$-\cet\ (respectively $2n +1$-\cet). Let $M_{T}$ and $X_{T}$ be, respectively, the suspended compact surface and the vector field associated to  $T$. Call $p_0$ be the only singular point of $X_{T}$ (the restriction of $X_{T}$ to $S_T=M_{T} \setminus \{p_0\}$ is analytic). 

On account of Theorem~\ref{T:Cantor}, we are done with the proof if we are able to find a Cantor set in $S_{T^c}$, $\mathcal{K}$, containing $p_0$ and with the extra property that any other orbit of $X_T$ meeting $\mathcal{K}$ is dense backward and forwardly and meets $\mathcal{K}$ in exactly one point.

Let $k=n$ (respectively $k=2n+2$) and $a_1=0<a_2<\cdots<a_{k+1}=1$ be the discontinuity points of $T$. Call, for every $1 \leq i \leq k$, $\lambda_i=a_{i+1}-a_i$. Scaling the interval $[0,1]$ by an appropriate irrational number if necessary, there is no loss of generality in assuming that all the $\lambda_i$ are irrational numbers.
 
Let us consider a maximal rationally independent set $F=\{\lambda_{j_1}, \ldots, \lambda_{j_N}\} \subset \{\lambda_1, \ldots, 
\lambda_{n}\}$ (i.e. any $\lambda_i \notin F$ makes $F\cup\{ \lambda_i\}$ be rationally dependent). The previous observation guarantees $F$ is non-empty. We may also assume that $\lambda_{j_1}=\lambda_1$.

Use Lemma~\ref{L:cantor}  to take    a rationally independent Cantor set, $K \subset [0,1]$, containing $F$. We show that, for any 
$k \in\Z\backslash\{0\}$, $T^k(K) \cap K \neq \emptyset$. Indeed, assume by contradiction the existence of  $x,y \in K$ and $k \geq 1$ such that $T^k(x)=y$. Using 
Equation~\eqref{formulaiet} and the maximality of $F$, we deduce that there are some $n_x,n_y, n_1, \ldots, n_N \in \Z$, not all vanishing, such that $n_x x + n_y y + n_{1} \lambda_{j_1} + \cdots + n_N \lambda_{j_N} \in \Z$ contradicting the rational independence of $K$. An analogous reasoning justifies that $\{\lambda_1\}= K \cap \bigcup_{i=1}^{n}{\mathcal{O}(a_i)}$. If we then identify again $0$ and $1$, $K$ can be seen as Cantor set in $\S^1$. Finally, call $\mathcal{K}=[K \times \{1\}] \subset S_{T^c}$ and observe that $p_0=[(\lambda_1,1)] \in \mathcal{K}$ and that any other orbit of $X_T$ meeting $\mathcal{K}$ does it exactly once and it is dense backward and forwardly. This completes the proof of Theorem~\ref{T:maingenerofinito}.


\section*{Acknowledgements}
This work is supported in part by the MINECO grants MTM2011-23221 and
MTM2014-52920-P and the ICMAT--Severo Ochoa grant
SEV-2015-0554. J.G.E.B. is supported by Fundaci\'on S\'eneca through the program ``Contratos Predoctorales de Formaci\'{o}n
del Personal Investigador'', grant 18910/FPI/13. D.P.-S. is supported by the ERC Starting Grant~335079. 

Part of this work was done during two visits of J.G.E.B. to the ICMAT (on the fall 2015 and the spring 2016); this author is
deeply grateful for the hospitality he received.

\appendix

\section{There are no minimal nonorientable surfaces of genus $3$}\label{ApenA}
In \cite[p. 14]{rusos}, the impossibility of finding minimal $C^{\infty}$ vector fields on nonorientable surfaces of genus $3$ is stated. We devote this appendix to show a proof of this fact as an elementary consequence of Lemma~\ref{ConnectedCircles} below, due to  C. Gutierrez (see~\cite{gutierrez2}).
 
Let $X$ be a $C^{\infty}$ complete vector field on a surface $S$ and, for every $p \in S$, let $\gamma_p:\R \to S$ be the integral curve of $X$ starting at $p$. We denote $\Gamma_p=\gamma_p(\R)$. The semiorbits $\Gamma_p^+=\{\gamma^p(t): t\in [0,\infty)\}$ and $\Gamma_p^-=\{\gamma^p(t): t\in (-\infty,0]\}$ are called, respectively, the \textit{positive} and the \textit{negative semiorbits} of $X$ \textit{starting at $p$}.  The  \emph{$\omega$-limit set of $p$} and the \emph{$\alpha$-limit set of $p$} are, respectively, $\omega_X(p)=\bigcap_{t \geq 0} \Cl(\{\gamma_p(s) : s \geq t\})$ and $\alpha_X(p)=\bigcap_{t \leq 0} \Cl(\{\gamma_p(s) : s \leq t\})$. A point $p\in S$ whose associated orbit meets either is $\alpha$-limit set or its $\omega$-limit set is said to be a recurrent point of $X$. Clearly, if $p$ is either a singular point of $X$ or $\gamma_p$ is periodic, $p$ is a recurrent point. Those points are also called the trivial recurrent points of $X$.

A subsurface with boundary $N\subset S$ is said to be a \emph{flow box} 
of $X$  if there exists an homeomorphism $\theta:[-1,1]\times [-1,1] \to N$ such that, 
for any $s\in[-1,1]$, $\theta([-1,1]\times \{s\})$ is a semiorbit of  $X$; in such a scenery, $\theta$ will be also called a flow box. 
We remind that any non singular point $p\in S$ is contained in a flow box. Even more, according to~\cite[Theorem~1.1, p. 45]{rusos}:

\begin{lemma}[Long Flow Box Theorem]\label{LongBox} Let $p$ be a regular point of $X$ and $\Gamma \subset \Gamma_p=\gamma(\R)$ be a semiorbit which is not closed. Then there exists a flow box $\theta:[-1,1]\times [-1,1] \to W$ such that $\gamma \subset \theta((-1,1)\times\{0\})$.
\end{lemma}

A circle $C \subset S$ is said to be \textit{transverse} to $X$ if for any $p\in C$ there exists a flow box $\theta:[-1,1]\times[-1,1]\to N$ such that $\theta(0,0)=p$ and $\theta(\{0\}\times [-1,1])=N\cap C$. Equivalently (c. f. \cite[p. 312]{gutierrez2}), a circle $C \subset M$ may be said to be a transverse to $X$ if there exists some $\epsilon>0$ such that the map $(t,p) \mapsto \gamma_p(t)$ is a homeomorphism of $[-\epsilon, \epsilon] \times C$ onto the closure of an open neighbourhood of $C$.

Let $C$ be a transverse circle. By definition of transversality, $C$ is orientable. Let $z \in S$ be such that $\gamma_z$ meets $C$ at least twice. Let $t_p < t_q$ be such that $\gamma_z(t_p)=p$ and $\gamma_z(t_q)=q$ are two different points of $C$ and $C \cap \gamma_z(t_p,t_q)) = \emptyset$. Let $L_{p,q}$ and $L'_{p,q}$ be the two components of $C\setminus\{p,q\}$ (both $L_{p,q}$ and $L'_{p,q}$ are open arcs with $p$ and $q$ as endpoints) and consider the circles $J_{p,q}=\gamma_z([t_p,t_q]) \cup L_{p,q}$ and $J'_{p,q}=\gamma_z([t_p,t_q]) \cup L'_{p,q}$. 
These two circles are said to be the \emph{two $C$-arcs determined by $\gamma_z$ and the points $p$ and $q$} (we shall also say that they are a \emph{pair of conjugated $C$-arcs determined by $\Gamma_z$}). Since $C$ is orientable, either both $J_{p,q}$ and $J'_{p,q}$ are orientable or both of them are nonorientable.
%
The following two Lemmas corresponds, respectively, with~\cite[Lemma 2]{gutierrez2} and~\cite[Proposition~2]{gutierrez2}.

\begin{lemma}[Peixoto]
\label{TransverseCircle} Let $S$ be a compact surface and $X$ be a complete $C^{\infty}$ vector fields. If $\gamma_p$ is an integral curve starting at a nontrivial recurrent point of $X$, then there is a transverse circle $C$ meeting with $\gamma_p$.
\end{lemma}

\begin{lemma}[Gutierrez]\label{ConnectedCircles} Let $X$ be a $C^{\infty}$ complete vector field on $N_3$ and let $z\in N_3$ be a regular point such that $\gamma_z$ is not periodic and $\Gamma_z \subset \omega_{X}(z)$. If $\Gamma_w$ is the orbit starting at a point $w$ and different of $\Gamma_z$  with the property that $\Gamma_z \subset \omega_{X}(w)$ and $C$ is a transverse circle to $X$ meeting $\Gamma_z$, then every $C$-arc determined by $\Gamma_w$, with possible exception  of one conjugated pair of them, is orientable.  
\end{lemma}

Finally, let $S$ be a nonorientable surface of genus $3$. According to the classification of noncompact surfaces (see Section~\ref{basic}), it is not restrictive to assume that $S\subset N_3$  and  $T=N_3 \setminus S$ is totally disconnected.

Let us proceed by contradiction assuming that there is a $C^{\infty}$ vector field on $S$, $X$,  with all its orbits being dense. It is not restrictive to assume that $X$ is defined on the whole $N_3$ having as critical points all the points in $T$. Thus, every orbit of $X$ starting at a point in $S$ is dense in $N_3$.

Without restriction of generality, we can take two points $p,q \in S$ whose orbits are different and both positively dense. Let $C$ be a transverse circle to $X$ (Lemma~\ref{TransverseCircle}). Since $\Gamma_p^+$ and $\Gamma_q^+$ are dense, they meet infinitely many times $C$. In particular there exists an increasing sequence $0 \leq t_1 < t_2 < t_3 < \dots$ such that $C \cap \Gamma_q^+ = \{ q_n=\gamma_q(t_n) : n \in \N \}$. For every $n \in \N$, call $\Gamma_n=\{ \gamma_q(t) : t_n \leq t \leq t_{n+1}\}$, a closed arc of endpoints $q_n$ and $q_{n+1}$. Let $J_{n}$ and $J'_{n}$ be the two $C$-arcs determined by $\Gamma_q$ and the points $q_n$ and $q_{n+1}$. In view of Lemma~\ref{ConnectedCircles}, either all the circles $J_n$ are orientable or all of them are orientable with possible a unique exception. In the latter case, say $J_{m}$ is the exception. Let $M$ be either $\mathbb{N}$ when all the $J_n$ are orientable or $\mathbb{N}\setminus\{m\}$ otherwise. 

Let $\epsilon>0$ be such that the restriction of the map $(t,x) \mapsto \gamma_x(t)$ to $[-\epsilon,\epsilon]\times C$ gives an embedding with $C=\Phi(\{0\} \times C)$ in the interior of $V=\{\gamma_x(t) : (t,x) \in [-\epsilon,\epsilon]\times C \}$. For every $n \in M$, consider also a flow box $\theta_n:[-1,1]\times[-1,1] \to W_n$ with $\Gamma_n \subset \theta_n((-1,1)\times\{0\})$ (see Lemma~\ref{LongBox}) and call $U_n=W_n \cup V$; this $U_n$ is then an orientable 
neighbourhood of $J_n$. 
As each $\theta_n$ is a flow box is clear that if, for some $1 \leq i,j \leq n$, $\Inte W_i \cap \Inte W_{j} \neq \emptyset$, then $W_i \cup W_j \cup V$ is orientable as well. 
As a consequence, the union of all the $U_n$ for $n \in M$ is orientable. 

Finally, since $\gamma_q$ is positively dense, the union $\bigcup_{n \in M}{U_n}$ is orientable and dense in $S$ and the orientation can be extended to the whole $S$. Indeed, for every $w\in S \setminus \bigcup_{n \in M}{U_n}$ let us take $s_w<0<t_w$ such that $\Gamma_w=\gamma_w([s_w,t_w]) \cap C = \{\gamma_w(s_w), \gamma_w(t_w)\}$ and a flow box $\theta_w:[-1,1]\times[-1,1] \to W_w$ with $\Gamma_w \subset \theta_w((-1,1)\times\{0\})$. Associated to any of these $w\in S \setminus \bigcup_{n \in M}{U_n}$, the positive density of $\gamma_q$ guarantees the existence of integers $n$ such that $\Gamma_n \subset W_w$; since $W_n \cup V$ is orientable then so must be $W_w \cup V$. Consequently we find around any point in $S \setminus \bigcup_{n \in M}{U_n}$ a neighbourhood (in $S$) with an orientation compatible with the one of $\bigcup_{n \in M}{U_n}$, a contradiction with the  nonorientability of $S$.

\section{Proof of Lemma~\ref{L:cantor}}\label{ApenB}
Let $I$ be an open interval and $F\subset I$ be finite and rationally independent. We shall construct, by induction, a nested sequence of compact sets whose intersection gives a Cantor set $K$ such that $F \subset K \subset I$.

Given any $p_1, \ldots, p_j \in \mathbb{Z}$ we name $f_{p_1, \ldots, p_j}$ the polynomial $f_{p_1, \ldots, p_j}(y_1, \ldots, y_j)= \sum_{i=1}^j{p_i y_i}$.

If $G \subset \mathbb{R}$ is a finite subset (say $G$ has exactly $m$ elements), the set $H=\{\sum_{i=1}^m{p_i y_i}: p_i \in \mathbb{Z}, \, y_i \in G\}$ is infinite. This trivial observation implies that, when $G$ is rationally independent, all but countably many elements of every connected component of $\R\setminus H$ can be added to $G$ and generate a new finite rationally independent set.

Therefore, we can assume we have a finite sequence $a^0_{1} < a^0_{2} < a^0_{3} < \cdots a^0_{2m-1} < a^0_{2m}$ (in $I$), for some positive integer $m$, such that $F'=\{a^0_i : 1 \leq i \leq 2 m\}$ is rationally independent and $F=\{a^0_{2i} : 1 \leq i \leq m\}$. Call $K_0= \cup_{i=1}^m{B^0_i}$ with $B^0_i=[a^0_{2 i -1},a^0_{2 i}]$ ($1 \leq i \leq m$).

Since $F'$ is rationally independent, for every choice of integers $p_1, \ldots, p_{2m}$ in  $\mathbb{Z}$ which do not vanish simultaneously, $f_{p_1, \ldots, p_{2 m}}(a^0_1, \ldots, a^0_{2 m}) \notin \mathbb{Z}$. 
By continuity we can then guarantee the existence of $4 m$ real numbers, $a^1_{1} < \cdots < a^1_{4 m}$, with $a^1_{4 i -3}=a^0_{2 i -1}$ and $a^1_{4 i}=a^0_{2 i}$ for every $1 \leq i \leq m$ such that $f_{p_1, \ldots, p_{2 m}}(y_1, \ldots, y_{2 m}) \notin \mathbb{Z}$ whenever  $p_1, \ldots, p_{2m} \in \{-2m, \ldots, -1, 0, 1, \ldots, 2m \}$, not all zero, and $(y_1, \ldots, y_{2 m}) \in B^1_1 \times \cdots \times B^1_{2 m}$ where $B^1_i=[a^1_{2 i -1},a^1_{2 i}]$ ($1 \leq i \leq 2 m$). We also assume, changing the interval for smaller ones if needed, that $0< a^1_{2 i} - a^1_{2 i -1} < 1/(2 m)$. Call $K_1= \cup_{i=1}^{2m}{B^1_i}$

Proceeding recursively, we build a infinite nested sequence $K_1 \supseteq \cdots \supseteq K_n  \supseteq K_{n+1} \supseteq \cdots$ such that, for every positive integer $n$

\begin{itemize}
	\item  $K_n$ is the union of $2^n m$ compact connected intervals $B^n_{i}=[a^n_{2 i -1},a^n_{2 i}]$ ($1 \leq i \leq 2^{n} m$) with 
	\begin{equation}\label{eqppB} a^n_{2 i} - a^n_{2 i -1} < 1/(2^n m)\end{equation} and where $a^n_1 < \cdots <a^n_{2^{n+1} m}$ and $a^{n+1}_{4 i}=a^n_{2 i}$ and $a^{n+1}_{4 i -3}=a^n_{2 i -1}$ for every $1 \leq i \leq 2^n m$;
	\item $F \subset K_n$;
	\item  for every $p_1, \ldots, p_{2^n m} \in \{  -2^n m, \ldots, -1, 0, 1, \ldots, 2^n m \}$, not all simultaneously zero, and every  $(y_1, \ldots, y_{2^n m}) \in B^n_1 \times \cdots \times B^n_{2^n m}$
	\begin{equation}\label{eqappB} f_{p_1, \ldots, p_{2^n m}}(y_1, \ldots, y_{2^n m}) \notin \mathbb{Z}.\end{equation} 
\end{itemize} 

We now take $K=\cap_{n = 1}^{\infty}{K_n}$ and claim that $K$ is the desired Cantor set. Indeed, $K$ is compact, perfect and totally disconnected set of real numbers containing $F$. Firstly, $K$ is a compact set containing $F$ because it is an intersection of sets with that property; even more, $K$ contains the set $\{a^n_{i}: n \in \mathbb{N} \wedge 1 \leq i \leq 2^{n+1} m\}$. Secondly, $K$ is rationally independent. Indeed, for every finite sequences $y_1, \ldots, y_l \in K$ and $p_1, \ldots p_l \in \mathbb{Z}$ we can take a large enough integer $n$ such that $\left|p_j\right| \leq 2^n m$ for all $0 \leq j \leq l$ and apply \eqref{eqappB} to guarantee that $\sum_{j=1}^m{p_j y_j} \notin \mathbb{Z}$ unless $p_1=\cdots=p_l=0$. Thirdly, $K$ is totally disconnected because, using~\eqref{eqppB}, for any two different points $x,y \in K$ there exist some big enough $n$ and two different $i,i'$ such that $x \in B^n_i$ and $y \in B^n_{i'}$. And, finally, $K$ is perfect. Indeed, let $x \in K$
  and $\varepsilon >0$. Let $n$ be a large enough integer so $2^{n} m > 1/\varepsilon$, then $x \in B^n_{i}$ for some  $1 \leq i \leq 2^{n} m$ and therefore $a^n_{2i-1}$ and $a^n_{2i}$ are points in $K \cap (x - \varepsilon, x + \varepsilon)$.




\end{document}